\documentclass[onefignum,onetabnum,final]{siamart171218}

\usepackage{enumerate}
\usepackage{braket,amsfonts,amsopn}
\usepackage{graphicx,epstopdf} 
\usepackage{yhmath}
\usepackage{mathrsfs}
\usepackage{mathtools}
\usepackage{float}
\usepackage{subfigure}
\usepackage{makecell}
\usepackage{multirow}

\usepackage{bbm}

\newcommand{\avg}[1]{\left\langle #1 \right\rangle}
\newcommand{\bx}{\mathbf{x}}
\newcommand{\bu}{\mathbf{u}}
\newcommand{\bM}{\mathbf{M}}
\newcommand{\bp}{\mathbf{p}}
\newcommand{\bv}{\mathbf{v}}

\newcommand{\bi}{\mathbf{i}}

\newcommand{\dif}{\mathrm{d}}
\newcommand{\mC}{\mathcal{C}}
\newcommand{\mS}{\mathcal{S}}
\newcommand{\cT}{\mathcal{T}}

\makeatletter
\newcommand{\fdsy@scale}{1.0}
\newcommand\fdsy@mweight@normal{Book}%
\newcommand\fdsy@mweight@small{Regular}%
\newcommand\fdsy@bweight@normal{Medium}%
\newcommand\fdsy@bweight@small{Bold}%
\DeclareFontFamily{U}{FdSymbolF}{}
\DeclareFontShape{U}{FdSymbolF}{m}{n}{
    <-7.1> s * [\fdsy@scale] FdSymbolF-\fdsy@mweight@small
    <7.1-> s * [\fdsy@scale] FdSymbolF-\fdsy@mweight@normal
}{}
\DeclareFontShape{U}{FdSymbolF}{b}{n}{
    <-7.1> s * [\fdsy@scale] FdSymbolF-\fdsy@bweight@small
    <7.1-> s * [\fdsy@scale] FdSymbolF-\fdsy@bweight@normal
}{}
\makeatother
\DeclareSymbolFont{delimiters}{U}{FdSymbolF}{m}{n}
\SetSymbolFont{delimiters}{bold}{U}{FdSymbolF}{b}{n}
\DeclareMathDelimiter{\lAngle}{\mathopen}{delimiters}{"92}{delimiters}{"92}
\DeclareMathDelimiter{\rAngle}{\mathclose}{delimiters}{"98}{delimiters}{"92}
\newcommand{\bbrk}[1]{\left\lAngle #1 \right\rAngle}

\newcommand{\revise}{}

\usepackage{lipsum}
\usepackage{amsfonts}
\usepackage{graphicx}
\usepackage{epstopdf}
\usepackage{algorithmic}
\ifpdf
  \DeclareGraphicsExtensions{.eps,.pdf,.png,.jpg}
\else
  \DeclareGraphicsExtensions{.eps}
\fi

\newsiamremark{remark}{Remark}
\newsiamremark{hypothesis}{Hypothesis}
\crefname{hypothesis}{Hypothesis}{Hypotheses}
\newsiamthm{claim}{Claim}

\headers{
  High-order SLFV Methods for the Advection Equation}{Y. Sun, K. Liang, Y. Zhu, 
  Z. Lin, and Q. Zhang}

\title{Fourth- and Higher-Order Semi-Lagrangian Finite Volume Methods 
  for the Two-dimensional Advection Equation
  on Arbitrarily Complex Domains 
  \thanks{
    Yunxia Sun and Kaiyi Liang are co-first authors 
    with equal contributions.
    \funding{Q. Zhang was supported by the grant 12272346 
      from the National Natural Science Foundation of China.
      Z. Lin was supported by grants 12071429 and 12090020 
      from the National Natural Science Foundation of China.
} 
} 
}

\author{
  Yunxia Sun \thanks{College of Mathematics and System Science,
      Xinjiang University, Urumqi, Xinjiang Prov., 830046 China
      (\email{syxmath@163.com}, \email{qinghai@zju.edu.cn}).} 
  \and Kaiyi Liang \thanks{School of Mathematical Sciences,
      Zhejiang University, Hangzhou, Zhejiang Prov., 310058, China 
    (\email{kyliang@zju.edu.cn}, \email{yukezhu0323@126.com}, 
    \email{linzhi80@zju.edu.cn}, \email{qinghai@zju.edu.cn}).}
  \and Yuke Zhu \footnotemark[3]
  \and Zhi Lin \footnotemark[3]
  \and Qinghai Zhang \footnotemark[2] \footnotemark[3] \thanks{
    Corresponding author. (\email{qinghai@zju.edu.cn})
    Institute of Fundamental and Transdisciplinary Research,
    Zhejiang University, Hangzhou, Zhejiang Prov., 310058, China.
      }
}

\usepackage{amsopn}

\makeatletter
\newcommand*{\addFileDependency}[1]{
  \typeout{(#1)}
  \@addtofilelist{#1}
  \IfFileExists{#1}{}{\typeout{No file #1.}}
}
\makeatother

\ifpdf
\hypersetup{
  pdftitle={Fourth- and Higher-Order Semi-Lagrangian Finite Volume Methods 
  for the Two-dimensional Advection Equation
  on Arbitrarily Complex Domains}, 
  pdfauthor={Y. Sun, K. Liang, Y. Zhu, Z. Lin, 
    and Q. Zhang}
}
\fi

\begin{document}

\maketitle




\begin{abstract}
  To numerically solve the two-dimensional advection equation, 
   we propose a family of fourth- and higher-order
   semi-Lagrangian finite volume (SLFV) methods
   that feature
   (1) fourth-, sixth-, and eighth-order convergence rates, 
   (2) applicability to
   both regular and irregular domains
   with arbitrarily complex topology and geometry, 
   (3) ease of handling
   both zero and nonzero source terms, 
   and (4) the same algorithmic steps
   for both periodic and incoming penetration conditions.
  Test results confirm the analysis
  and demonstrate the accuracy, flexibility, robustness,
  and excellent conditioning of the proposed SLFV method.
\end{abstract}



\section{Introduction}
\label{sec:introduction}
The advection equation is 
 an important partial differential equation (PDE) 
 utilized across various scientific fields
 such as meteorology and fluid mixing and transport.  
This paper focuses on the two-dimensional advection equation
\begin{subequations}
  \label{eq:advection_control_equation}
  \begin{align}
    \frac{\partial \rho(\bx, t)}{\partial t} 
        + \bu(\bx, t) \cdot \nabla \rho(\bx, t) 
        &=S(\bx, t),
    &&\text{in}\ \Omega \times [0, T], 
    \label{eq:advection_control_equation_a}\\
    \label{eq:advection_control_equation_init}
    \rho(\bx, 0) &= \rho_{\text{init}}(\bx),  
    &&\text{in}\ \overline{\Omega}, \\  \label{eq:advection_control_equation_bc}
    \rho(\bx, t) &= \rho_{\text{bc}}(\bx, t),  
    &&\text{on}\ \Gamma^-, 
  \end{align}  
\end{subequations}
where $\bx$ is the spatial position, 
$t$ the time, 
$\rho$ the unknown scalar,  
$\Omega\subset \mathbb{R}^2$ a bounded domain (i.e., a bounded Yin
set; see \Cref{def:YinSets}),
$\partial \Omega$ the domain boundary, 
$\overline{\Omega}$ the closure of $\Omega$, 
$\rho_{\text{init}}$ the initial condition, 
$\rho_{\text{bc}}$ the boundary condition, 
$\bu: \overline{\Omega}\times[0,T]\to \mathbb{R}^2$
 a velocity field given \emph{a priori}, 
and $S(\bx, t)$ the source term
that is continuous in $t$
and Lipschitz continuous in $\bx$. 
If $\rho$ is periodic on $\Omega$,
 \cref{eq:advection_control_equation_bc}
 should be removed from \cref{eq:advection_control_equation}; 
 otherwise, to ensure the well-posedness of
 \cref{eq:advection_control_equation}, 
 $\rho_{\text{bc}}$ 
 should be specified on a subset
 $\Gamma^-\subset \partial \Omega \times (0, T]$
 that depends on the sign of $\mathbf{u}$ along $\partial\Omega$, 
\begin{equation}
  \label{eq:incomingSet}
  \Gamma^- \coloneqq\bigl \{(\bx,t): \bx\in \partial \Omega,\, t\in (0,T]; \, 
  \mathbf{u}(\bx,t) \cdot \mathbf{n}(\bx)<0\bigr\},  
\end{equation} 
 where $\mathbf{n}(\bx)$ is the unit outward normal vector
 of $\partial\Omega$.
If $\Gamma^-\ne\emptyset$,
 we call $\rho_{\text{bc}}: \Gamma^- \to \mathbb{R}$
 an \emph{incoming penetration condition}
 and $\Omega$ an \emph{incoming penetration domain}
 (or simply a \emph{penetration domain}). 
If $\Gamma^-=\emptyset$,
 \cref{eq:advection_control_equation_bc}
 drops out from \cref{eq:advection_control_equation}
 without affecting the well-posedness;
 see the last paragraph of \Cref{sec:flow-maps}
 for such an example.
 
Numerical algorithms for solving the advection equation 
 can be categorized into Eulerian, Lagrangian, 
 and semi-Lagrangian methods.  
These methods can be conceptually unified via
 \emph{the Lagrangian form of the advection equation}, 
\begin{equation}
  \label{eq:lagrangian_derivative}
  \frac{\dif \rho}{\dif t} = S(\bx, t),
\end{equation}
 which is equivalent to \cref{eq:advection_control_equation_a}
 since the material derivative is defined
 as $\frac{\dif}{\dif t} \coloneqq \frac{\partial}{\partial t} 
    + \bu \cdot \nabla $.
In the extended phase space,
 integrate \cref{eq:lagrangian_derivative}
 along a curve
 $\gamma:[s_0,s_1] \to \Omega \times [0,T]$ 
 and we have
 $\rho(\bx_1, t_1) - \rho(\bx_0, t_0) 
    = \int_{\gamma} \dif \rho 
    = \int_{\gamma}\left(
        \nabla \rho, \frac{\partial \rho}{\partial t}
        \right) \cdot (\dif \bx, \dif t)$, i.e., 
\begin{equation}
  \label{eq:Eulerian_Lagrangian}
  \rho(\bx_1, t_1) 
    = \rho(\bx_0, t_0) 
    + \int_{\gamma} \nabla \rho \cdot (\dif \bx - \bu \dif t) 
    + \int_{\gamma} S \ \dif t, 
\end{equation}
which yields a pure Eulerian method
if $\gamma$ is chosen to be time independent
and a pure Lagrangian method
 if $\gamma$ is an integral curve of \cref{eq:ODE-flow-map}
 or the pathline of a Lagrangian particle; see \Cref{def:pathline}.  
Another choice of $\gamma$ as a combination 
 of stationary curves and particle pathlines
 leads to the so-called \emph{semi-Lagrangian methods}
 \cite{staniforth1991semi, smolarkiewicz1992class}. 

The time step size of pure Eulerian methods is constrained by
 the Courant-Friedrichs-Lewy (CFL) stability condition. 
On the other hand, the time step size of a pure Lagrangian method
 can be arbitrarily large, 
 but particle distributions may evolve to be highly non-uniform,
 which deteriorates accuracy.  
Semi-Lagrangian methods resolve disadvantages 
 of the two pure methods
 by employing a regular Cartesian grid for uniform spatial distribution 
 and tracking particles along pathlines 
 to relax the CFL condition; 
 they have proven to be 
 effective in various applications such as 
 weather forecasting \cite{robert1981stable, staniforth1991semi, 
   mohebalhojeh2009diabatic}
 and plasma simulations \cite{begue1999two, sonnendrucker1999semi,
    besse2003semi, huot2003instability, carrillo2007nonoscillatory,
    qiu2010conservative}. 

In semi-Lagrangian methods, 
 the unknown may be point values for finite difference schemes 
 \cite{qiu2010conservative, qiu2011conservative, 
   xiong2019conservative, chen2021adaptive, li2023high} 
 or cell-averaged values for finite volume schemes 
 \cite{crouseilles2010conservative, huang2012eulerian, 
   benkhaldoun2015family, abreu2017new}. 
In both cases, 
 the solution often needs to be determined
 from nearby values
 via a reconstruction technique such as 
 spline interpolation \cite{crouseilles2007hermite}, 
 piecewise parabolic method (PPM) \cite{colella1984piecewise}, 
 cubic interpolation pseudo-particle method (CIP) 
 \cite{takewaki1985cubic},
 and essentially non-oscillatory (ENO)/ weighted ENO (WENO) reconstruction
 \cite{jiang1996efficient, shu2009high, shu2020essentially}.

Many semi-Lagrangian methods 
 were originally designed in one dimension; 
 see, e.g., the finite difference scheme in \cite{qiu2011conservative}
 and the finite volume scheme in \cite{huang2012eulerian}.
Multi-dimensional problems are often decomposed
 into a series of one-dimensional subproblems
 via an operator splitting technique such as
 the second-order Strang splitting \cite{qiu2011positivity} 
 and the fourth-order splitting in \cite{nakao2022eulerian}. 

Recently, 
 non-splitting semi-Lagrangian schemes are becoming increasingly popular.
For example,
 the consideration of multi-dimensional flux differences
 leads to a mass conservative semi-Lagrangian finite difference WENO scheme 
 \cite{xiong2019conservative}, 
 which imposes an additional restriction 
 on the time step size for numerical stability. 
Another conservative WENO scheme \cite{zheng2022fourth}
 involves constructing a new WENO reconstruction operator 
 and computing the integral over cell preimages. 
A machine learning-assisted method \cite{chen2023learned}
 accelerates traditional schemes 
 by learning the semi-Lagrangian discretization from data,
 but is unfortunately subject to the stability constraint 
 of the CFL number being less than 2.

Despite their tremendous successes,
 current semi-Lagrangian methods still have a number of limits.
First, 
 most methods are developed for regular domains
 and it is not clear how to extend them to irregular domains 
 with complex geometry. 
Second,
 most methods assume a periodic condition on $\rho$
 and it is not clear how to generalize these methods to
 penetration domains.
Third,
 although it encourages reuse of simple algorithms, 
 the operator splitting technique may result in a significant increase
 of the number of subproblems in multiple dimensions. 
For a targeting order of accuracy,
 one may have to re-derive the splitting process, 
 which may impose additional constraints
 on the time step size.

The above discussions lead to the following questions:
\begin{enumerate}[({Q}-1)]
\item \label{Q-1}
  Can we develop fourth- and higher-order semi-Lagrangian methods 
  that effectively handle both regular and irregular domains 
  with arbitrarily complex topology and geometry?   
\item \label{Q-2}
  Can the new semi-Lagrangian method
  apply to both periodic domains and penetration domains?
\item \label{Q-3} 
  Is the new semi-Lagrangian method directly applicable
  to the case of a nonzero source term? 
\item \label{Q-4}
  While semi-Lagrangian methods are free of the CFL constraints, 
  can we prove the fourth- and higher-order convergence rates 
  of the new method?
\end{enumerate}

In this paper, 
 we give positive answers to all above questions 
 by proposing a family of fourth- and higher-order
 semi-Lagrangian finite volume (SLFV) methods. 

The rest of this paper is organized as follows.
We set the theoretical context for this work
 in \Cref{sec:preliminaries},
 where we prove \cref{thm:partitionOfYinSets}
 as the key result for handling irregular domains. 
In \Cref{sec:algorithm}, 
 we describe our answers to (Q-1,2,3)
 by elaborating on the SLFV method
 and its main components. 
The convergence rates of SLFV are proved in \Cref{sec:analysis},
 answering (Q-4). 
In \Cref{sec:numerical-tests},
 we perform various tests on SLFV
 to demonstrate its high-order accuracy,
 its capability of handling irregular domains 
 with arbitrarily complex topology and geometry,
 and its generality for periodic/penetration domains
 and zero/nonzero source terms.
Finally, 
 we conclude this work in \Cref{sec:conclusion-1}
 with several research prospects.

\section{The action of flow maps upon Yin sets}
\label{sec:preliminaries}
Continua of arbitrarily complex topology and geometry
 are modeled by Yin sets \cite{zhang2020boolean} 
 and the action of flow maps upon Yin sets
 furnishes the theoretical context of this paper. 

\subsection{Flow maps}
\label{sec:flow-maps}

The ordinary differential equation (ODE)
\begin{equation}
  \label{eq:ODE-flow-map}
  \frac{\dif \bx}{\dif t} = \bu(\bx, t)
\end{equation}
has a unique solution for any given initial time $t_0$ 
and position $\mathbf{p}(t_0)$ 
if $\bu(\bx, t)$ is continuous in time 
and Lipschitz continuous in space. 
This uniqueness gives rise to a \emph{flow map} 
$\phi:\mathbb{R}^{2} \times \mathbb{R} \times \mathbb{R}
    \to \mathbb{R}^{2}$, 
which takes the initial time $t_0$, 
the initial position $\mathbf{p}(t_0)$ 
of a Lagrangian particle $\mathbf{p}$, 
and the time increment $\pm k$,  
and returns the position of $\mathbf{p}$ at $t_0 \pm k$:
\begin{equation}
  \label{eq:flow-maps}
  \phi_{t_0}^{\pm k}(\mathbf{p}) 
  \coloneqq \mathbf{p}(t_0 \pm k)
  = \mathbf{p}(t_0) 
  + \int_{t_0}^{t_0 \pm k} \bu(\mathbf{p}(t),t)\ \dif t.
\end{equation}

The set $\{\phi_{t_0}^{\sigma}: \sigma\in \mathbb{R}\}$
forms a one-parameter group of diffeomorphisms.

\begin{definition}[Pathline]
  \label{def:pathline}
  The \emph{pathline} of a Lagrangian particle $\mathbf{p}$ 
  in a time interval $(t_0,t_0 \pm k)$
  is the curve $\Phi_{t_0}^{\pm k}: (0,k)\to \mathbb{R}^2$ given by 
  \begin{equation}
    \label{eq:pathline}
    \Phi_{t_0}^{\pm k}(\bp) \coloneqq \{ \phi_{t_0}^{\pm \tau}(\bp) : \tau \in (0, k)\}.
  \end{equation}
\end{definition}

The flow map naturally generalizes
 to a point set $\mathcal{P}$. 
As a homeomorphism, the flow map preserves
key topological features of $\mathcal{P}$.

In the particular case of a \emph{no-penetration domain}, 
i.e.,
  $\forall \mathbf{x}\in\partial \Omega,\ 
  \mathbf{u}(\mathbf{x}) \cdot \mathbf{n}(\mathbf{x}) = 0$, 
 the flow map sends points in $\Omega$
 to points in $\Omega$
 and thus a no-penetration domain $\Omega$ contains
 the pathline of any particle $\bp$
 that is initially inside $\Omega$.

\subsection{Yin sets and their partitions}
\label{sec:Yin-sets}
Denote by $\mathcal{X}$ a topological space 
and $\mathcal{P} \subseteq \mathcal{X}$ a subset.  
The complement, closure and interior of $\mathcal{P}$ 
are denoted as $\mathcal{P}'$, $\overline{\mathcal{P}}$ 
and $\mathcal{P}^{\circ}$, respectively. 
The \emph{exterior} of $\mathcal{P}$ is defined as 
$\mathcal{P}^{\perp} \coloneqq (\mathcal{P}')^{\circ}$. 
$\mathcal{P} \subseteq \mathcal{X}$ is \emph{regular open} 
if $\mathcal{P} = \left(\mathcal{P}^{\perp}\right)^{\perp}$. 
This regularity condition captures
the physical meaningfulness of continua
by precluding low-dimensional features
such as isolated gaps and points.

A set $\mathcal{P} \subseteq \mathbb{R}^{2}$ is \emph{semianalytic} 
if there exist a finite number of analytic functions 
$g_i: \mathbb{R}^{2} \to \mathbb{R}$ such that 
$\mathcal{P}$ is in the universe of a finite Boolean algebra 
formed by the sets 
$\{\bx \in \mathbb{R}^{2} : g_i(\bx) \ge 0\}$.  
Intuitively, semianalytic sets have piecewise smooth boundary curves.

\begin{definition}[Yin space~\cite{zhang2020boolean}]
  \label{def:YinSets}
  A \emph{Yin set} $\mathcal{Y} \subseteq \mathbb{R}^{\text{2}}$ 
  is a regular open semianalytic set whose boundary is bounded. 
  The class of all such Yin sets form the \emph{Yin space} 
  $\mathbb{Y}$.
\end{definition}

The boundedness of the boundary of Yin sets leads to  

\begin{theorem}[Zhang and Li~\cite{zhang2020boolean}]
  \label{thm:BooleanAlgebra}
  $\left(\mathbb{Y}, \cup^{\perp \perp}, \cap, ^{\perp}, 
    \emptyset, \mathbb{R}^2\right)$ is a Boolean algebra, 
  where the regularized union $\cup^{\perp \perp}$
  is given by $\mathcal{P}\cup^{\perp \perp}\mathcal{Q}
    \coloneqq (\mathcal{P} \cup \mathcal{Q})^{\perp \perp}$ 
  for all $\mathcal{P},\ \mathcal{Q}\in \mathbb{Y}$.
\end{theorem}

Efficient algorithms of Boolean operations on Yin sets 
 have been developed in \cite{zhang2020boolean}
 and are employed in this work
 to cut \revise{and merge} cells.

The condition of Yin sets being open leads to the conclusion
\cite[Theorem 3.9]{zhang2020boolean}
that the boundary $\partial {\mathcal Y}$
 of any connected Yin set ${\mathcal Y}\ne \emptyset, \mathbb{R}^2$
 can be \emph{uniquely} decomposed into a finite set
 ${\mathcal J}_{\partial{\mathcal Y}}$
 of pairwise almost disjoint Jordan curves.
 
\begin{theorem}
  \label{thm:uniqueCases}
  Any connected Yin set ${\mathcal Y}\ne \emptyset, \mathbb{R}^2$
  can be uniquely expressed as
  \begin{equation}
    \label{eq:orientedJordanCurves}
    {\mathcal Y} = \bigcap_{\gamma_j\in {\mathcal J}_{\partial
        {\mathcal Y}}} \mathrm{int}(\gamma_j), 
  \end{equation}
  where $\mathrm{int}(\gamma_{j})$ is the complement of $\gamma_{j}$ 
  that always lies to the left of the observer who traverses $\gamma_{j}$
  according to its orientation
  and ${\mathcal J}_{\partial {\mathcal Y}}$ 
  must be one of the two types,
  \begin{equation}
    \label{eq:decomTypes}
    \renewcommand{\arraystretch}{1.2}
    \left\{
      \begin{array}{ll}
        {\mathcal J}^-
        =\{\gamma^-_1, \gamma^-_2, \ldots, \gamma^-_{n_-}\},
        & n_-\ge 1,
        \\
        {\mathcal J}^+
        =\{\gamma^+,\gamma^-_1, \gamma^-_2, \ldots, \gamma^-_{n_-}\}, 
        & n_-\ge 0,
      \end{array}
    \right.
  \end{equation}
  with $\gamma^+$ being positively oriented
  and all $\gamma^-_j$'s being
  negatively oriented and pairwise almost disjoint.
  In the second case,
  we have $\gamma_j^- \prec \gamma^+$
  for each $j=1,2,\ldots,n_-$, i.e.,
  the bounded complement of each $\gamma_j^-$
  is a subset of the bounded complement of $\gamma^+$.
\end{theorem}
\begin{proof}
  See \cite[page 2348]{zhang2020boolean}. 
\end{proof}

This unique boundary representation extends
 to all nontrivial Yin sets 
 via appropriately orienting the boundary Jordan curves, i.e., 
 any $\mathcal{Y} \ne \emptyset,\mathbb{R}^2$
 can be \emph{uniquely} expressed \cite[Corollary
 3.13]{zhang2020boolean} as 
 $\mathcal{Y} = \sideset{}{^{\perp \perp}}\bigcup_j \bigcap_i
  \mathrm{int}(\gamma_{j,i})$
  where $j$ is the index of connected components of $\mathcal{Y}$
  and $\gamma_{j,i}$ the $i$th oriented Jordan curve
  for the $j$th component.

Following \cite{hu2025ARMS},
 we approximate oriented Jordan curves by cubic splines.
Hereafter $\mathbb{Y}_c$ denotes the class of all Yin sets
 whose boundaries consist of cubic splines.
It is straightforward to show that $\mathbb{Y}_c$ is a subspace and a subalgebra of $\mathbb{Y}$.

A Yin set $\mathcal{Y} \in \mathbb{Y}_c$ is called
 a \emph{simple splinegon} 
 if it is homeomorphic to an open disk, 
 i.e., if it is categorized as the second case $\mathcal{J}^+$ 
 in \cref{eq:decomTypes} with $n_-=0$. 
A \emph{splinegon with holes} is a Yin set in $\mathbb{Y}_c$
 categorized as the second case $\mathcal{J}^+$ with $n_{-} > 0$.

\begin{theorem}
  \label{thm:partitionOfYinSets}
  Each bounded Yin set $\mathcal{Y} \subset \mathbb{R}^2$ 
   can be partitioned into a regularized union 
   of a finite number of curved quadrilaterals and triangles.
\end{theorem}
\begin{proof}
  It suffices to prove the statement
  for bounded and connected Yin sets in $\mathbb{Y}_c$.
  The boundedness of $\mathcal{Y}$ precludes the first case 
   $\mathcal{J}^{-}$ in \cref{eq:decomTypes} 
   and \cref{thm:uniqueCases} dictates that 
   $\mathcal{Y}$ be represented as $\mathcal{J}^+$. 
  By \cref{eq:decomTypes}, we only need to consider two possibilities.  
  \begin{itemize}
  \item $n_{-} = 0$: $\mathcal{Y}$ is a simple splinegon
    homeomorphic to an open disk. 
    If $\mathcal{Y}$ has no more than four vertices, 
    the conclusion holds trivially
    since it is already a curved quadrilateral or triangle.
    Otherwise we map $\mathcal{Y}$ by a homeomorphism $\varphi$
     to an open disk,
     partition the disk by connecting its center
     to every other vertex on the circle,
     and the images of these radii under $\varphi^{-1}$
     partition $\mathcal{Y}$ into  
     a regularized union 
     of curved quadrilaterals and triangles.
    By \cref{def:YinSets},
     $\mathcal{Y}$ is semianalytic and thus the number of
     quadrilaterals is finite.
    As shown in \cref{fig:simplePolygonToDisk},
     the number of triangles is at most one: 
     the triangle exists if and only if
     the number of vertices of ${\cal Y}$ is odd.
   \item $n_{-} > 0$: $\mathcal{Y}$ is a splinegon with holes.
     As shown in \cref{fig:Jordan_polygon_with_holes}, 
      we add a pair of ``bridge'' edges 
      from a vertex of each $\gamma_j^{-}$
      to some vertex in $\gamma^+$ of $\mathcal{J}^+$
      in \cref{eq:decomTypes}; 
      the existence of these bridge edges
      is guaranteed by the connectedness of ${\cal Y}$.
      Now that ${\cal Y}$ is approximated arbitrarily well
      by a simple splinegon, 
      the rest of the proof follows from the previous paragraph.
    \end{itemize} 
\end{proof}

\begin{figure}
  \centering
  \includegraphics[width=0.98\textwidth]
    {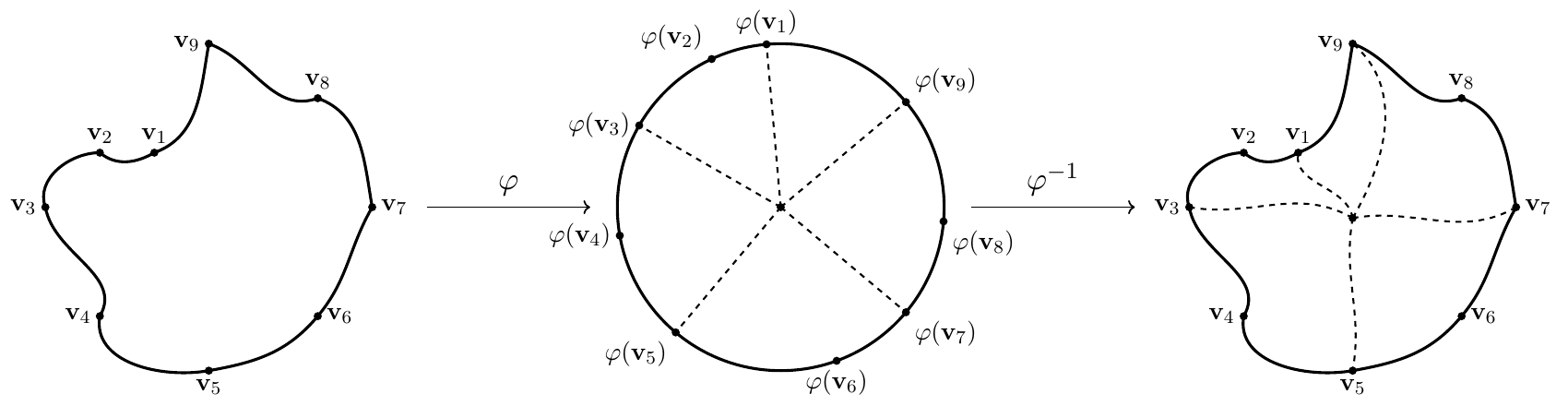} 
  \caption{$\mathcal{Y}$ is homeomorphic to an open disk 
   in the case of $n_{-} = 0$. 
   The dashed straight lines represent the division of the disk, and
   the dashed curves represent the division of $\mathcal{Y}$.}
  \label{fig:simplePolygonToDisk}
\end{figure}

\begin{figure}
  \centering
  \subfigure[Creating ``bridge'' edges.]{
    \includegraphics[width=0.4\textwidth, height=40mm]
      {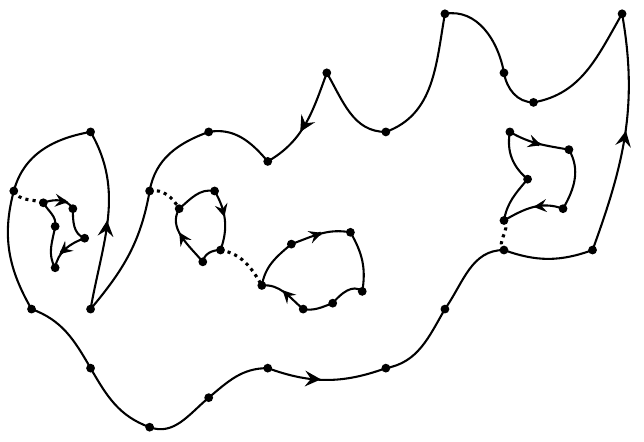}
    \label{fig:Jordan_polygon_with_holes_1}
  }
  \subfigure[The resulting simple splinegon.]{
    \includegraphics[width=0.4\textwidth, height=40mm]
      {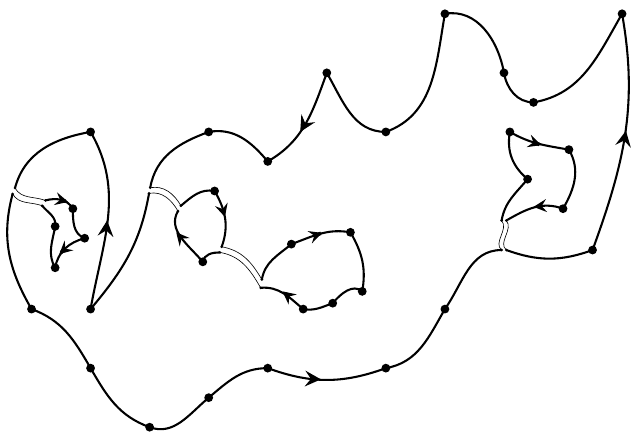}
    \label{fig:Jordan_polygon_with_holes_2}
  } 
  \caption{Transforming a splinegon with holes into a simple splinegon.}  
  \label{fig:Jordan_polygon_with_holes}
\end{figure}

\section{Algorithm}
\label{sec:algorithm}
We embed the problem domain $\Omega$
 inside a rectangular region
 $\Omega_R \in \mathbb{Y}_c$ 
 and divide $\Omega_R$ by structured rectangular grids 
 into control volumes or cells, 
 \begin{equation}
   \label{eq:controlVolume}
   \mathbf{C}_\bi \coloneqq 
   \left(\bi h, (\bi+\mathbbm{1})h \right), 
 \end{equation}
 where $h$ is the uniform grid size, 
 $\bi\in\mathbb{Z}^2$ a multi-index, 
 and $\mathbbm{1}\in\mathbb{Z}^2$ the multi-index 
 with all components equal to one.
The assumption of $h$ being uniform
 is for ease of exposition only, 
 as our algorithm also applies to
 non-uniform grids.
In both cases, we have
$\Omega_R = \cup^{\perp\perp}_{\bi} \mathbf{C}_\bi$
where $\cup^{\perp\perp}$ is the regularized union operation
in \cref{thm:BooleanAlgebra}. 
 
For an irregular domain $\Omega$, 
we define the $\bi$th \emph{cut cell}
and the $\bi$th \emph{cut boundary} as
\begin{equation}
  \label{eq:cutCellAndBdry}
  \mC_\bi \coloneqq \mathbf{C}_\bi \cap \Omega;  \quad
  \mS_\bi \coloneqq \mathbf{C}_\bi \cap \partial \Omega,
 \end{equation}
 where $\cap$ is the intersection operation
 in \cref{thm:BooleanAlgebra}. 
Clearly we have $\Omega = \cup^{\perp\perp}_{\bi} \mC_\bi$. 

A cut cell is classified as 
 an \emph{empty cell}, a \emph{pure cell},
 or an \emph{interface cell} if 
 it satisfies $\mC_\bi = \emptyset$, 
 $\mC_\bi = \mathbf{C}_\bi$,
 or otherwise, respectively.

The averaged values of a scalar function $g:\Omega\to \mathbb{R}$ 
 over a nonempty cut cell $\mC_\bi$
 and a cut boundary $\mS_\bi$ 
 are respectively,
\begin{equation}
  \label{eq:cellAndFaceAverage}
  \avg{g}_\bi \coloneqq 
  \frac{1}{\|\mC_\bi\|} \iint_{\mC_\bi} g(\bx) \ \dif \bx;
  \quad
  \bbrk{g}_\bi \coloneqq 
  \frac{1}{\|\mS_\bi\|} \int_{\mS_\bi} g(\bx)\ \dif \bx,
\end{equation}
where $\|\mC_{\bi}\|$ is the volume of $\mC_{\bi}$
and $\|\mS_\bi\|$ the length of $\mS_\bi$. 

The time interval $[0,T]$ is partitioned
 into $N_T$ subintervals 
 with a uniform time step size $k=\frac{T}{N_T}$
 so that $t_n \coloneqq nk$ for $n=0,\ldots,N_T$.
To numerically solve the advection equation
\cref{eq:advection_control_equation}, 
it suffices to approximate $\avg{\rho(\bx,t_{n+1})}_\bi$
 for an arbitrary cut cell $\mC_\bi$
 from the given velocity field $\mathbf{u}$,
 the initial condition
 $\avg{\rho}^{n}_\bi\approx \avg{\rho(\bx,t_{n})}_\bi$, 
 and the boundary condition of $\rho$.

\begin{definition}[SLFV]
  \label{def:SLFV}
  The \emph{SLFV method} is 
  a finite volume method for solving the advection equation
  \cref{eq:advection_control_equation}
  by approximating $\avg{\rho(\bx,t_{n+1})}_\bi$
  with steps as follows.
   \begin{enumerate}[(SLFV-1)]
   \item For each cut cell $\mC_{\bi}\subset \Omega$
     and any scalar function $g:\mC_{\bi}\to \mathbb{R}$,
     determine a quadrature formula 
     \begin{equation}
       \label{eq:quadrature}
       I_{\bi}(g) \coloneqq 
       \sum\nolimits_{m=1}^M \omega_{\bi,m} g(\bx_{\bi,m})
     \end{equation}
     by calculating the nodes $\{\bx_{\bi,m}\}_{m=1}^M$
     and the weights $\{\omega_{\bi,m}\}_{m=1}^M$ of $I_{\bi}$
     from \cref{eq:Gauss-quadrature}
     in \Cref{sec:Gauss-quadrature-formula}
     and the partitioning of $\mC_{\bi}$
     in \cref{thm:partitionOfYinSets}.
   \item For each node $\bx_{\bi,m}$,
     compute a set of points
     $\{\bx_{\bi,m}^{n+c_i}\approx \phi_{t_{n+1}}^{-(1-c_{i})k}(\bx_{\bi,m})\}$
     on the pathline $\Phi_{t_{n+1}}^{-k}(\bx_{\bi,m})$
     by solving the ODE \cref{eq:ODE-flow-map}. 
     The approximate preimage 
     of $\bx_{\bi,m}$ at $t_n$ is denoted by
     $\overleftarrow{\bx_{\bi,m}}\coloneqq  \bx_{\bi,m}^{n+0}$;
     \revise{see \Cref{fig:GaussPointsBacktracking}}.
   \item If $\overleftarrow{\bx_{\bi,m}}\not\in \Omega$,
     calculate
     $\tilde{\rho}^{n+1}(\bx_{\bi,m})\approx \rho(\bx_{\bi,m},t_{n+1})$
     by \cref{eq:RKfromIntersection} in \Cref{sec:penetration}. 
   \item Otherwise approximate $\rho(\overleftarrow{\bx_{\bi,m}},t_n)$
     by $f(\overleftarrow{\bx_{\bi,m}})$,
     the value of a local bi-variate polynomial $f$
     fitted from the initial condition $\avg{\rho}^n_\bi$
     and the boundary condition of $\rho$ at $t_n$; 
     see \Cref{sec:reconstruction} for an elaboration
     of this process.
     Then calculate
     $\tilde{\rho}^{n+1}(\bx_{\bi,m})
     \approx \rho(\bx_{\bi,m},t_{n+1})$
     by solving \cref{eq:lagrangian_derivative}
     with an $s$-stage Runge-Kutta (RK) method, 
     \begin{equation}
       \label{eq:RK}
       \begin{aligned}
         \forall i=1,\ldots,s,\ \  
         Y_i &= 
         S\left(
           \bx_{\bi,m}^{n+c_i}, \ 
           t_n + c_i k
         \right), \\
         \tilde{\rho}^{n+1}(\bx_{\bi,m}) &= 
         f(\overleftarrow{\bx_{\bi,m}}) + 
         k\sum\nolimits_{j=1}^s b_j Y_j,
       \end{aligned}
     \end{equation}
     where $b_j,\ c_i$ are coefficients of the RK method
     and $\bx_{\bi,m}^{n+c_i}$'s points
     of the pathline $\Phi_{t_{n+1}}^{-k}(\bx_{\bi,m})$ obtained in (SLFV-2).
   \item Repeat (SLFV-2,3,4) for each quadrature node
     $\bx_{\bi,m}$ in \cref{eq:quadrature} 
     and set the solution of SLFV for the cut cell ${\cal C}_{\bi}$
     as $\avg{\rho}^{n+1}_\bi\coloneqq I_{\bi}(\tilde{\rho}^{n+1})
     \approx \avg{\rho(\bx,t_{n+1})}_\bi$.
   \end{enumerate}
 \end{definition}

\revise{
\begin{figure}
  \centering
  \subfigure[\revise{Calculating preimages of Gauss quadrature points 
  inside a cut cell $\mC_{\bi}$.
  In particular, the preimage of $\partial\mC_{\bi}$ is not constructed.}
  ]{ 
    \includegraphics[width=0.45\textwidth]
      {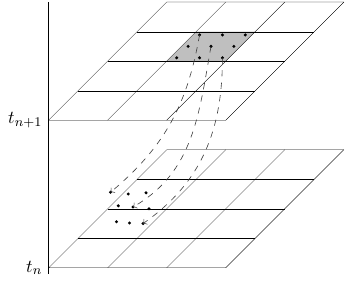}
    \label{fig:GaussPointsBacktracking} 
  }
  \hspace{1cm}
  \subfigure[\revise{Intersecting the pathline with the domain boundary
  if the preimage of $\mathbf{x}_{\bi,m}$ is not in $\Omega$.}]{
    \includegraphics[width=0.3\textwidth]
      {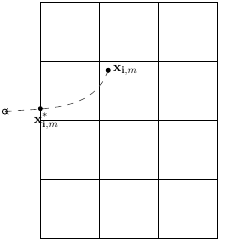}
    \label{fig:OnePointBacktracking}
  } 
  \caption{The backtracking of Gauss quadrature points inside a
    control volume. 
    The gray quadrilateral represents the cut cell $\mC_{\bi}$ 
     for which the averaged value of $\rho$ at time $t_{n+1}$ is sought.
    A dashed line with an arrowhead represents a pathline 
     obtained by numerically solving the ODE \cref{eq:ODE-flow-map}.
    (a) The solid dots at $t_{n+1}$ represent the Gauss quadrature points
     while those at $t_n$ denote their preimages.
     If the preimage of a quadrature point $\mathbf{x}_{\bi,m}$ is not in $\Omega$, 
     the pathline of $\mathbf{x}_{\bi,m}$ must intersect the domain boundary
     at some time $t^*\in (t_n, t_{n+1})$.
    (b) $\bx_{\bi,m}$ and $\bx_{\bi,m}^*$ represent 
     the Gauss quadrature point at time $t_{n+1}$ and 
     the intersection of its pathline with the domain boundary at $t^*$, 
     respectively.
     The hollow circle denotes the preimage of $\bx_{\bi, m}$.
     }
  \label{fig:backtracking}
\end{figure}
}

The above steps of the SLFV method are orthogonal
 in that the subproblems of complex geometry,
 temporal integration, and polynomial reconstruction
 are decoupled. 
Consequently,
 fourth- and higher-order accuracy both in time and in space
 can be achieved in a flexible manner. 
In addition, the steps of SLFV remain the same
 regardless of the source term $S$ being zero or not; 
 this answers (Q-\ref{Q-3}).

\revise{
\begin{remark}
  In a family of conservative semi-Lagrangian methods
  \cite{lauritzen2010conservative,crouseilles2014new,einkemmer2024semi},
  the preimage of the boundary of a control volume
  is constructed
  and its intersection with nearby control volumes is computed. 
  In our SLFV method, however, 
   this intersection is avoided
   by calculating  the preimages
   of Gauss quadrature points;
   see \cref{fig:GaussPointsBacktracking}.
  Only for handling incoming penetration domains
   would we compute the intersection of 
   a particle pathline and the domain boundary;
   see \cref{fig:OnePointBacktracking}. 
\end{remark}
} 

In the rest of this section,
 we fully explain the key components of SLFV.

\subsection{Gauss quadrature on interface cells}
\label{sec:Gauss-quadrature-formula}

Since $\partial \Omega$ is of co-dimension one, 
 most nonempty cut cells are pure cells, 
 for which Gauss quadrature formulas can be obtained
 by tensor products of one-dimensional formulas
 and will not be discussed further. 
Here we design a three-step quadrature formula for interface cells. 
\begin{enumerate}[({GQI}-1)]
\item Partition the interface cell $\mC_{\bi}$ 
  into a union of curved quadrilaterals and triangles; 
  see \cref{thm:partitionOfYinSets}
  and \cref{fig:simplePolygonToDisk} for such an algorithm.
\item \label{GQI-2}
  Calculate integrals on each curved quadrilateral or triangle
  via an isoparametric mapping and the blending function method
  \cite[p.99]{szabo2021finite}.
\item Sum up the subintegrals to approximate the integral over $\mC_{\bi}$.
\end{enumerate}

The first part of (GQI-2) is an isoparametric mapping
 $\bM(\xi,\eta)=\sum_{i=1}^4 N_i(\xi,\eta)\bv_i$
 where $\bv_i$'s are vertices of a quadrilateral and
 the Lagrangian shape functions are
\begin{equation*}
  \begin{aligned}
    N_1(\xi,\eta) = \frac{1}{4}(1-\xi)(1-\eta), \quad &
    N_2(\xi,\eta) = \frac{1}{4}(1+\xi)(1-\eta), \\
    N_3(\xi,\eta) = \frac{1}{4}(1+\xi)(1+\eta), \quad &
    N_4(\xi,\eta) = \frac{1}{4}(1-\xi)(1+\eta).
  \end{aligned}
\end{equation*}
As shown in \cref{fig:standardQuadrilateral}, 
 $\bM$ maps $[-1,1]^2$ to the given quadrilateral.
Hereafter we denote by
$\mathbf{E}_i\coloneqq(E_{i,x},E_{i,y})$ 
the parametrization of the $i$th (linear or curved) edge
 $\wideparen{\bv_{i}\bv_{i+1}}$
 of a quadrilateral
 with the convention $\bv_5=\bv_1$. 

\begin{figure}
  \centering
  \includegraphics[width=0.6 \textwidth]
    {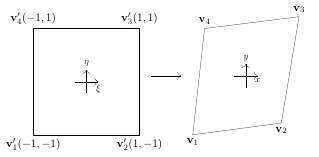}
  \vspace{-15pt}
  \caption{A mapping from $[-1,1]^2$ to a given quadrilateral with
    linear edges.}
  \label{fig:standardQuadrilateral}
\end{figure}

\begin{figure}[h!]  
  \centering
  \includegraphics[width=0.7 \textwidth]
    {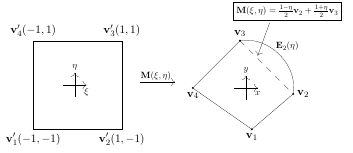} 
  \vspace{-15pt}
  \caption{A mapping from $[-1,1]^2$ to a quadrilateral 
    with one curved edge.}
  \label{fig:BlendingFunctionMethod}
\end{figure}

Next, we illustrate the main idea of the blending function method
 by the simple case in the right-hand side (RHS) of
 \cref{fig:BlendingFunctionMethod}, 
 where only one edge of the quadrilateral is curved.
The canonical square $[-1,1]^2$ is mapped to the curved quadrilateral
 by 
\begin{equation}
  \label{eq:blendingFunctionMapping2nd}
  \begin{aligned}
  \bM(\xi, \eta) =
  & \frac{1}{4}(1-\xi)(1-\eta)\bv_1 +
    \frac{1}{4}(1+\xi)(1-\eta)\bv_2 +
    \frac{1}{4}(1+\xi)(1+\eta)\bv_3 \\
  &+\frac{1}{4}(1-\xi)(1+\eta)\bv_4 + 
    \left(
      \mathbf{E}_2(\eta) - 
      \frac{1-\eta}{2}\bv_2 -
      \frac{1+\eta}{2} \bv_3
    \right) \frac{1+\xi}{2}, 
  \end{aligned}
\end{equation}
 where the first four terms represent 
 the linear isoparametric mapping 
 and the last a product of two functions: 
 the one in the parentheses represents the difference 
 between $\mathbf{E}_2(\eta)$ and 
 the chord joining $\bv_2$ and $\bv_3$
 while the other is a linear function 
 that equals $1$ on $\wideparen{\bv_2\bv_3}$ 
 and 0 on $\overline{\bv_4\bv_1}$. 
The mapping in \cref{eq:blendingFunctionMapping2nd} simplifies to 
\begin{equation}
  \label{eq:reBlendingFunctionMapping2nd}
  \bM(\xi, \eta) = 
  \frac{1}{4}(1-\xi)(1-\eta)\bv_1 + 
  \frac{1}{4}(1-\xi)(1+\eta)\bv_4 + 
  \mathbf{E}_2(\eta)\frac{1+\xi}{2}.
\end{equation}

The above mapping generalizes in a straightforward manner to the general case 
 where all edges of the quadrilateral are curved: 
\begin{equation}
  \label{eq:belndingFunctionMapping}
  \begin{aligned}
    \bM(\xi,\eta) = 
    & \frac{1}{2}\left[
        (1-\eta)\mathbf{E}_1(\xi) + 
        (1+\xi)\mathbf{E}_2(\eta) + 
        (1+\eta)\mathbf{E}_3(\xi) +
        (1-\xi) \mathbf{E}_4(\eta) 
      \right] \\
    &-\sum\nolimits_{i=1}^4 N_i(\xi,\eta) \bv_i. 
  \end{aligned}
\end{equation}

\begin{figure}[h!]  
  \centering
  \includegraphics[width=0.3\textwidth]
    {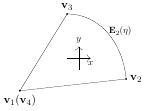} 
  \vspace{-15pt}
  \caption{A curved triangle with one curved edge.}
  \label{fig:BlendingFunctionForTriangle}
\end{figure}

A curved triangle is treated as a special quadrilateral
 with one edge collapsed into a single vertex. 
In the case of \cref{fig:BlendingFunctionForTriangle}, 
 the mapping in \cref{eq:reBlendingFunctionMapping2nd}
 reduces to 
\begin{equation}
  \label{eq:Mapping4Triangle-1}
  \bM(\xi, \eta) = \frac{1-\xi}{2}\bv_1 +
  \frac{1+\xi}{2}\mathbf{E}_2(\eta), 
\end{equation}
 whose Jacobian matrix is 
\begin{equation*}
  J = \begin{pmatrix}
    -\frac{1}{2}x_1+\frac{1}{2}E_{2,x}(\eta) &
    \frac{1+\xi}{2}E_{2,x}'(\eta) \\
    -\frac{1}{2}y_1+\frac{1}{2}E_{2,y}(\eta) &
    \frac{1+\xi}{2}E_{2,y}'(\eta)
  \end{pmatrix}.
\end{equation*}
As $\xi$ approaches $-1$, 
 the determinant $|J|$ of the Jacobian maxtrix approaches zero, 
indicating a singularity at $\bv_1$. 
Fortunately,
 it is well known that
 all nodes of a Gauss formula are contained in the integral domain
 and hence, for any fixed degree of exactness,
 their distances to the integral boundary 
 are bounded below by a positive real number. 
Thus $|J|$ also has a positive lower bound that never approaches zero 
 and the mapping in \cref{eq:Mapping4Triangle-1} remains valid
 for sending the quadrature nodes in $[-1,1]^2$
 to those inside the triangle. 

For a triangle whose edges are all curved, 
the mapping in \cref{eq:belndingFunctionMapping} 
 reduces to 
\begin{equation}
  \label{eq:Mapping4Triangle-2}
  \begin{aligned}
    \bM(\xi,\eta) = 
    & \frac{1}{2}\left(
        (1-\eta)\mathbf{E}_1(\xi) + 
        (1+\xi)\mathbf{E}_2(\eta) + 
        (1+\eta)\mathbf{E}_3(\xi)
      \right) \\
    &-N_2(\xi,\eta)\bv_2 - N_3(\xi,\eta)\bv_3.
  \end{aligned}
\end{equation}

To sum up, the Gauss formula
 of approximating the integral of a scalar function $g$ 
 over a curved quadrilateral or a curved triangle is summarized as 
\begin{equation}
  \label{eq:Gauss-quadrature}
  I(g) \coloneqq 
  \sum\limits_{i'=1}^{m_\xi} \sum\limits_{j'=1}^{m_\eta} 
  \omega_{i',j'}g(\bM(\xi_{i'},\eta_{j'})),
\end{equation}
where $\bM$ is given
by \cref{eq:belndingFunctionMapping}
and \cref{eq:Mapping4Triangle-2}
for a curved quadrilateral and a curved triangle,
respectively, 
$\{\xi_{i'}\}_{i'=1}^{m_\xi}$, 
$\{\eta_{j'}\}_{j'=1}^{m_\eta}$ 
and $\{\omega_{i'}\}_{i'=1}^{m_\xi}$, 
$\{\omega_{j'}\}_{j'=1}^{m_\eta}$ 
are the nodes and weights of the one-dimensional Gauss formula
in the $\xi$ and $\eta$ directions, respectively, 
and $\omega_{i',j'} \coloneqq 
  \omega_{i'}\omega_{j'} |J(\xi_{i'},\eta_{j'})|$.

Our Gauss formula (GQI-1,2,3)
 answers (Q-\ref{Q-1}) since it
 applies to 
 irregular domains 
 with arbitrarily complex topology and geometry.

\subsection{Handling incoming penetration domains}
\label{sec:penetration}
Recall from \cref{eq:incomingSet} that
 a/an (incoming) penetration domain is a non-periodic domain
 that satisfies $\Gamma^-\ne \emptyset$.
In this case the pathline $\Phi_{t_{n+1}}^{-k}(\bx_{\bi,m})$
 constructed in (SLFV-2)
 might intersect $\partial \Omega$. 
The existence of such an intersection is implied
 by the condition $\overleftarrow{\bx_{\bi,m}}\not\in\Omega$
 in \mbox{(SLFV-3)}, 
 which can be detected 
 by shooting a ray from $\overleftarrow{\bx_{\bi,m}}$
 and verifying that the number of intersections of the ray
 to $\partial \Omega$ is even.

Suppose the pathline $\Phi_{t_{n+1}}^{-k}(\bx_{\bi,m})$ intersects $\partial \Omega$ at
 $\bx_{\bi,m}^*$ and $t^*\in (t_n,t_{n+1})$, 
 \revise{as shown in \cref{fig:OnePointBacktracking},} 
 we can approximate $\rho(\bx_{\bi,m},t_{n+1})$
 from the boundary condition $\rho_{\mathrm{bc}}(\bx_{\bi,m}^*,t^*)$
 by solving \cref{eq:lagrangian_derivative}
 along a shortened pathline $\Phi_{t^*}^{t_{n+1}-t^*}(\bx_{\bi,m}^*)$
 with $\rho_{\mathrm{bc}}(\bx_{\bi,m}^*,t^*)$ as its initial condition.

To determine $\bx_{\bi,m}^*$ and $t^*$, 
 we define a function $\mathbf{F}(s,t)\coloneqq 
 \gamma(s) - \phi_{t_{n+1}}^{t-t_{n+1}}(\bx_{\bi,m})$
 where $s\in[0,L]$
 and $\gamma\coloneqq (\gamma_x,\gamma_y)$
 is a local arc of $\partial\Omega$
 that covers $\bx_{\bi,m}^*$. 
Then we solve the nonlinear equation $\mathbf{F}(s^*,t^*)=(0,0)$
 with the Newton iteration,  
\begin{equation*}
  \begin{pmatrix}
    s^{(\iota+1)} \\ t^{(\iota+1)}
  \end{pmatrix} = 
  \begin{pmatrix}
    s^{(\iota)} \\ t^{(\iota)}
  \end{pmatrix} - 
  \left(
    \frac{\partial \mathbf{F}}{\partial (s,t)}
    \left(s^{(\iota)},t^{(\iota)}\right)
  \right)^{-1} 
  \mathbf{F}\left(s^{(\iota)},t^{(\iota)}\right)^\top, 
\end{equation*} 
where the Jacobian matrix,
by the flow map $\phi$ in \cref{eq:flow-maps}, is 
\begin{equation*}
  \frac{\partial \mathbf{F}}{\partial (s,t)}(s,t) = 
  \begin{pmatrix}
    \gamma_x'(s) & -u\left(
      \phi_{t_{n+1}}^{t-t_{n+1}}(\bx_{\bi,m}),t
    \right) \\
    \gamma_y'(s) & -v\left(
      \phi_{t_{n+1}}^{t-t_{n+1}}(\bx_{\bi,m}),t
    \right)
  \end{pmatrix},
\end{equation*}
\revise{and $\mathbf{u}=(u,v)^\top$}. 

The iteration starts with $t^{(0)}=t_{n+1}$ 
and $\gamma(s^{(0)})$, the point in $\partial\Omega$
closest to $\bx_{\bi,m}$, 
and continues until 
$\left\|\mathbf{F}\left(s^{(\iota)},t^{(\iota)}\right)\right\|_2$ 
falls below a user-defined threshold 
or the maximum number of iterations is reached. 
 
After determining $t^*$ and $\bx_{\bi,m}^*$, 
 we set $k^* \coloneqq t_{n+1} - t^*$
 and solve a modified version of \cref{eq:RK}
 with $t^*$ as the initial time 
 and $\rho_{\text{bc}}(\bx_{\bi,m}^*,t^*)$ as the initial value, 
 \begin{equation}
   \label{eq:RKfromIntersection}
    \begin{aligned}
    \forall i=1,\ldots,s,\ \  
      Y_i &= 
      S\left(
        \bx_{\bi,m}(t^* + c_i k^*), \ 
        t^* + c_i k^*
      \right), \\
       \tilde{\rho}^{n+1}(\bx_{\bi,m}) &= 
        \rho_{\text{bc}}(\bx_{\bi,m}^*, t^*) + 
        k^*\sum\nolimits_{j=1}^s b_j Y_j.
    \end{aligned}
\end{equation}

The above steps constitute
 an efficient and accurate treatment
 of incoming penetration conditions, 
 answering (Q-\ref{Q-2}).

\subsection{Polynomial reconstruction from cell averages}
\label{sec:reconstruction}

We start with polynomial fitting of cell averages.

\begin{definition}
  Denote by $\Pi_q$ the linear space
  of all real-coefficient bi-variate polynomials 
  of degree no greater than $q$ 
  and write $N\coloneqq\mathrm{dim}\Pi_q$.
  For a finite class of pairwise disjoint Yin sets 
   $\mathcal{T} \coloneqq \{\mC_{\bi_1},\ldots,\mC_{\bi_N}\}$
   and a corresponding set of values
   $f_1,\ldots,f_N\in \mathbb{R}$, 
   the \emph{Lagrange interpolation problem (LIP) for cell averages}
   seeks a polynomial $f\in\Pi_q$ such that
   $\avg{f}_{\bi_j}$, the average of $f$ over $\mC_{\bi_j}$, equals
   $f_j$ for each $j$, i.e., 
   \begin{equation}
     \label{eq:LIP}
     \forall j=1,\ldots,N, \quad \avg{f}_{\bi_j}=f_j.
   \end{equation}
\end{definition}

A LIP is said to be \emph{unisolvent} 
 if for \emph{any} given sequence of cell averages $(f_j)_{j=1}^N$ 
 there exists $f\in\Pi_q$ satisfying \cref{eq:LIP}; 
 then we also call
 the set $\mathcal{T}$ a \emph{poised lattice}. 

\begin{figure}
  \centering
  \subfigure[The standard stencil for pure cells
  far from the domain boundary.]{
    \includegraphics[width=0.29\linewidth]
    {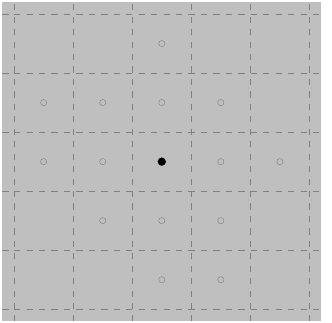}
    \label{fig:5th-regular-PLG}
  }
  \hfill
  \subfigure[A lattice for a pure cell near the domain boundary.]{
      \includegraphics[width=0.29\linewidth]
        {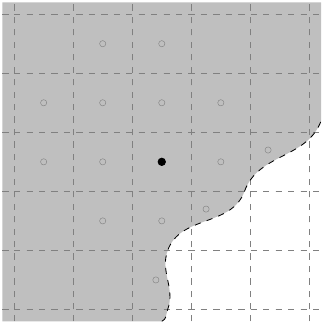}
      \label{fig:5th-purecell-PLG}
  } 
  \hfill
  \subfigure[A lattice for an interface cell near the domain boundary.]{
      \includegraphics[width=0.29\linewidth]
        {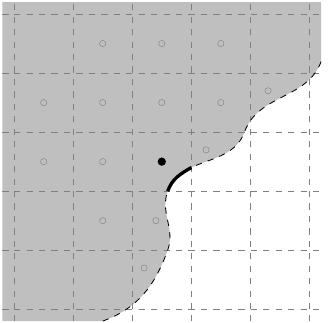}
      \label{fig:5th-irregular-PLG}
  }
  \caption{Poised lattices for the fifth-order polynomial reconstruction
    from cell averages.
    The gray region represents the computational domain. 
    The solid dot marks the target cut cell $\mC_\bi$, 
    while solid and hollow dots together represent the poised lattice. 
    The solid curve represents the additional boundary portion 
    within the target cell.
    \revise{The three subplots cover all cases.}
  } 
\end{figure}

For a pure cell far from the domain boundary, 
 we use the standard poised lattice
 shown in \cref{fig:5th-regular-PLG}.
\revise{
 Whenever the stencil in \cref{fig:5th-regular-PLG}
 involves an empty cell, 
 we employ the AI-aided algorithm in \cite{zhang2024PLG} 
 to generate a poised lattice.
 \cref{fig:5th-purecell-PLG} and \cref{fig:5th-irregular-PLG}
 demonstrate poised lattices generated by this AI-aided algorithm
 for a pure cell near the domain boundary
 and an interface cell, respectively.
With an optimal complexity,
 this AI-aided algorithm obviates the need for ghost cells
 at an irregular boundary;
 see \cite[Fig. 3]{zhang2024PLG} for more details. 
}
As shown in \cref{fig:5th-irregular-PLG}, 
 the cut boundary $\mS_\bi$
 of an interface cell $\mC_\bi$
 is also added into 
 the generated lattice 
 to incorporate the boundary condition.
Then we define
\begin{equation*}
  \boldsymbol{\rho} 
  \coloneqq [\rho_1,\ldots,\rho_N,\rho_b]^\top
  = \left[
      \avg{\rho}_{\bi_1},\ldots,\avg{\rho}_{\bi_N},
      \bbrk{\rho}_{\bi}
    \right]^\top \in \mathbb{R}^{N+1}.
\end{equation*}
Denote by $\boldsymbol{\alpha}\coloneqq[\alpha_1,\ldots,\alpha_N]^\top$ 
 coefficients of a polynomial 
 $f(\bx)=\sum_{i=1}^N \alpha_i\psi_i(\bx)\in\Pi_q$, 
 where $\psi_i$'s are the monomial bases of $\Pi_q$. 
To obtain a $(q+1)$th-order approximation of $\rho(\bx,t_n)$, 
 we determine $\boldsymbol{\alpha}$ by minimizing
\begin{equation*}
  \begin{split}
  & \sum\nolimits_{j=1}^N w_j \left|\avg{f}_{\bi_j} - \rho_j\right|^2 + 
   w_b \left|\bbrk{f}_{\bi} - \rho_b\right|^2 \\
  =& \sum\nolimits_{j=1}^N w_j \left(
    \sum\nolimits_{i=1}^N \alpha_i \avg{\psi_i}_{\bi_j}-\rho_j
   \right)^2 + w_b\left(
    \sum\nolimits_{i=1}^N \alpha_i \bbrk{\psi_i}_{\bi}-\rho_b
   \right)^2 \\
   =& (\boldsymbol{\alpha}^{\top} M - \boldsymbol{\rho}^{\top})W
   (M^\top\boldsymbol{\alpha} - \boldsymbol{\rho})
   = 
   \left\|M^\top\boldsymbol{\alpha} - \boldsymbol{\rho}\right\|_W^2,
  \end{split}
\end{equation*}
where $W$ is a positive definite diagonal matrix with 
 $W\coloneqq\mathrm{diag}(w_1,\ldots,w_N,w_b)$, 
 the \emph{energy norm} is given by
 $\|\mathbf{v}\|_W \coloneqq \sqrt{\mathbf{v}^{\top} W \mathbf{v}}$,
 and 
 \begin{equation}
   \label{eq:matrixM}
  M = 
  \begin{pmatrix}
    \avg{\psi_1}_{\bi_1} & \cdots & \avg{\psi_1}_{\bi_N} & \bbrk{\psi_1}_\bi \\ 
    \vdots & \ddots & \vdots & \vdots \\
    \avg{\psi_N}_{\bi_1} & \cdots & \avg{\psi_N}_{\bi_N} & \bbrk{\psi_N}_\bi
  \end{pmatrix}
  \in\mathbb{R}^{N \times (N+1)}.
\end{equation}
In this work, we
 set $w_j\coloneqq\min\{\|\bi_j-\bi\|_2^{-1}, w_{\max}\}$, 
 $w_b\coloneqq w_{\max}$, 
 and $w_{\max}=2$ to avoid large weights. 

For interface cells,
 the above minimization leads to a problem of weighted least squares, 
 which we solve by QR factorization. 
For pure cells,
 the last column of $M$ in \cref{eq:matrixM} is omitted, 
 and the square linear system $M^\top\boldsymbol{\alpha}=\boldsymbol{\rho}$
 is solved directly for the polynomial coefficients.
In both cases,
 we obtain a polynomial 
 $f(\bx)=\sum_{i=1}^N \alpha_i\psi_i(\bx)$
 by which we can evaluate point values of $\rho$
 at the preimages of nodes of Gauss formulas.

 \revise{
   Finally, we give some details
   on the stability of our polynomial reconstruction. 
  \begin{enumerate}[(a)]
  \item The total degree of the bi-variate polynomial is limited to 8
    to prevent wild oscillations of polynomial interpolation on uniform grid
    (the Runge phenomenon). 
  \item To avoid small volume fractions, 
    any cut cell that satisfies $\|\mC_{\bi}\| < \epsilon h^2$
    is merged into one of its adjacent cells
    so that the merged cell has a volume fraction
    no less than $\epsilon$. 
    Similarly,
    any interface cell with $\|\mS_{\bi}\| < \epsilon h$
    is merged into an adjacent cut cell
    so that the cut boundary has a length ratio
    no less than $\epsilon$. 
    In this work, 
    we set the user-specified merging threshold
    to $\epsilon = 0.1$, 
    which constraints the variation of volume fractions
    to one order of magnitude
    and guarantees good conditioning of the linear system
    for polynomial coefficients.
  \item The AI-aided algorithm for poised lattice generation 
    is originally developed for \emph{finite difference} formulations
    and also works well for \emph{finite volume} discretizations, 
    so long as each cell only contributes
    \emph{one} equation to the linear system.
    If we insist on $N$ equations 
    with an interface cell contributing two equations,
    one from the cell average and the other
    from the face average over the cut boundary,
    there might exist pathological cases that lead 
    to ill-conditioning or even singularity of the corresponding
    linear system. 
    It is also difficult to solve this problem analytically.
    Fortunately, a simple solution to this problem
    is to treat the equation from the cut boundary
    as one extra condition
    so that the over-determined linear system with $N+1$ equations
    contains a poised lattice of $N$ cells with good conditioning.
    Then QR factorization gives a well-conditioned solution
    of the least square problem.
  \item Apart from (c),
    additional points are sometimes appended into a poised lattice
    so that the condition number of the corresponding linear system
    is further lowered to a local minimum.
    In this type of least squares,
    the condition number is well under control
    via the number of additional points;
    see \cite[Fig. 8]{zhang2024PLG} for a few illustrations.
  \end{enumerate}
}

\section{Analysis}
\label{sec:analysis}
In this section, 
 we prove the convergence rates of the SLFV method. 
We start by examining Gauss formulas on curved quadrilaterals and triangles. 

\begin{lemma}
  \label{lem:accuracyAnalysisQuadrilateral}
  Denote by  
   $\mathbf{E}_i\coloneqq (E_{i,x},E_{i,y})$ 
   the $i$th edge of a curved quadrilateral or triangle
   ${\mathcal Q}\subset \mathbb{R}^2$.  
  Suppose that $E_{i,x}$ and $E_{i,y}$ 
   are both polynomials \mbox{$[-1,1]\mapsto \mathbb{R}$} of degree at most $q$.
  Then the quadrature formula in \cref{eq:Gauss-quadrature} is exact
   for any bi-variate polynomial
   \mbox{$g:{\mathcal Q}\to \mathbb{R}$} of degree at most $\kappa$ 
   in each variable
   provided that $\min(m_\xi,\ m_\eta) \ge \left\lceil
     \frac{(\kappa+2)q}{2}\right\rceil$
   holds for $m_\xi$ and $m_\eta$ in \cref{eq:Gauss-quadrature}.
\end{lemma}
\begin{proof}
    The variable substitution of $(x,y)=\mathbf{M}(\xi,\eta)$ in
    \cref{eq:Gauss-quadrature} yields 
    \begin{equation*}
      \iint_{{\mathcal Q}} g(x,y) \,\dif x \dif y
      = \int_{-1}^1\int_{-1}^1 
        g(\bM(\xi,\eta)) |J(\xi,\eta)|\,\dif\xi\dif\eta,
    \end{equation*}
     where $J(\xi,\eta)$ is the Jacobian matrix of $\bM(\xi,\eta)$. 
    Since $E_{i,x}(l)$ and $E_{i,y}(l)$ are both polynomials of degree at most $q$, 
     \cref{eq:belndingFunctionMapping} 
     and \cref{eq:Mapping4Triangle-2} imply that $\bM(\xi,\eta)$ is a polynomial 
     of degree at most $q$ in each variable. 
    Hence the composite function $g(\bM(\xi,\eta))$ is a polynomial 
     of degree at most $\kappa q$ in each variable
     and the determinant of the Jacobian matrix, $|J(\xi,\eta)|$,
     is a polynomial of degree at most $2q-1$ in each variable. 
    Thus the product $g(\bM(\xi,\eta))|J(\xi,\eta)|$ is a polynomial 
    of degree at most $(\kappa+2)q-1$ in each variable.
    Therefore, 
    for a one-dimensional Gauss formula to be exact, 
     it suffices to have
     $\min(m_\xi,\ m_\eta) \ge\left\lceil \frac{(\kappa+2)q}{2}\right\rceil$.
\end{proof}

\Cref{lem:accuracyAnalysisQuadrilateral}
is a key in analyzing the accuracy of the SLFV method.
\revise{
We also need 

\begin{lemma}
  \label{lem:well-posed}
  Consider the initial value problem (IVP)
  \begin{equation*}
    \mathbf{w}'(t)=\mathbf{f}(\mathbf{w}(t), t), \quad 
    \mathbf{w}(0)=\mathbf{w}_0,
  \end{equation*} 
   where $\mathbf{f}(\mathbf{w}, t)$ is Lipschitz continuous in $\mathbf{w}$ 
   and continuous in $t\in [0,T]$. 
  Let $\mathbf{v}(t)$ be the solution of the perturbed IVP over
  $[0,T]$, 
   \begin{equation*}
      \mathbf{v}'(t)=\mathbf{f}(\mathbf{v}(t), t) + \boldsymbol{\delta}(t), 
      \quad \mathbf{v}(0)=\mathbf{v}_0
      \coloneqq \mathbf{w}_0+\boldsymbol{\varepsilon}_0.
   \end{equation*}
  Suppose there exists a constant $\varepsilon>0$ such that 
   $\|\boldsymbol{\varepsilon}_0\|<\varepsilon$ 
   and $\|\boldsymbol{\delta}(t)\|<\varepsilon$, 
   then 
   \begin{equation*}
      \forall t\in[0, T], \quad 
      \|\mathbf{v}(t)-\mathbf{w}(t)\| \le 
      \left(1+Ct\right)\varepsilon, 
   \end{equation*}
   where $C\coloneqq 1+L\exp(LT)(1+T)$ 
   and $L$ is the Lipschitz constant of the RHS $\mathbf{f}$. 
\end{lemma}
\begin{proof}
  Since $\mathbf{v}(t)-\mathbf{w}(t)
    = \mathbf{v}_0-\mathbf{w}_0 + 
    \int_0^t [\mathbf{f}(\mathbf{v}(\tau), \tau) 
    - \mathbf{f}(\mathbf{w}(\tau), \tau) 
    + \boldsymbol{\delta}(\tau)]\,\dif \tau$,
    we have,
    from the triangular inequality 
    and the Lipschitz condition, 
  \begin{equation}
    \label{eq:perturbed_error_init}
     \begin{array}{l}
    \|\mathbf{v}(t) - \mathbf{w}(t)\|
    \le (1+t)\varepsilon + \int_0^t L\|\mathbf{v}(\tau)-\mathbf{w}(\tau)\|
        \,\dif \tau. 
     \end{array}
   \end{equation}
   
  To estimate the integral, 
   we define a function $g:[0,T]\rightarrow \mathbb{R}$, 
   \begin{equation*}
     \begin{array}{l}
       g(\tau) \coloneqq \exp(-L\tau)\int_0^\tau L\|\mathbf{v}(r)-\mathbf{w}(r)\|\,\dif r,
     \end{array}
  \end{equation*}
  so that
  $g'(\tau)=L\exp(-L\tau)\left(\|\mathbf{v}(\tau) - \mathbf{w}(\tau)\| 
      - \int_0^\tau L\|\mathbf{v}(r)-\mathbf{w}(r)\|\,\dif r\right)$.
  Since $g(0)=0$, we have $g(t) = g(t) - g(0) = \int_0^t g'(\tau)\,\dif \tau$.
   It follows that
   \begin{equation*}
     \begin{split}
       g(t)&=\exp(-L t)\int_0^t L\|\mathbf{v}(r) - \mathbf{w}(r)\|\,\dif r 
       \\ 
           &= \int_0^t L\exp(-L\tau)\left(\|\mathbf{v}(\tau) - \mathbf{w}(\tau)\| 
             - \int_0^\tau L\|\mathbf{v}(r) - \mathbf{w}(r)\|\,\dif r\right)\,\dif \tau\\
           &\le \int_0^t L\exp(-L\tau)(1+\tau)\varepsilon\,\dif \tau, 
     \end{split}
   \end{equation*}
   where the inequality follows from \cref{eq:perturbed_error_init}. 
   Multiply $\exp(L t)$ to the first and third lines
   and we have
   $\int_0^t L\|\mathbf{v}(\tau)-\mathbf{w}(\tau)\|\,\dif \tau 
   \le \int_0^t L\exp(L(t-\tau))(1+\tau)\varepsilon\,\dif \tau$,
   which, together with \cref{eq:perturbed_error_init}
   and the integral mean value theorem, yields
   \begin{displaymath}
     \begin{array}{l}
    \|\mathbf{v}(t)-\mathbf{w}(t)\|
    \le (1+t)\varepsilon + \int_0^t L\exp(L(t-\tau))(1+\tau)\varepsilon\,\dif \tau
    \le \left(1+Ct\right)\varepsilon.
     \end{array}
  \end{displaymath}
\end{proof} 
}
 
\begin{theorem}
  \label{thm:accuracyOfSLFV}
  Suppose the solution of \cref{eq:advection_control_equation}
   is sufficiently smooth
   and the domain $\Omega\in\mathbb{Y}$ is
   approximated by a Yin set $\Omega_c\in\mathbb{Y}_c$, 
   whose boundary curves consist of cubic splines
   \revise{of which the maximal chordal length
   of adjacent breakpoints is denoted by $h_L$}. 
  Then, \revise{in the max-norm,} the SLFV method in \Cref{def:SLFV} 
   \revise{with $k=O(h)$} is 
  \begin{itemize}
  \item fourth-order convergent both in time and in space 
     if \revise{we choose \mbox{$h_L = O(h)$},} 
     a 4th-order RK method
      in \cref{eq:RK}, \cref{eq:RKfromIntersection},
      and (SLFV-2), 
     a 5th-order polynomial reconstruction in (SLFV-4), 
     and Gauss formulas in \cref{eq:Gauss-quadrature} with 
     $m_\xi,\ m_\eta \ge 3$ for pure cells and
     $m_\xi,\ m_\eta \ge 9$ for interface cells.
   \item sixth-order convergent both in time and in space
     if \revise{we choose $h_L = O\left(h^{\frac{3}{2}}\right)$,}  
     a 6th-order RK method
      in \cref{eq:RK}, \cref{eq:RKfromIntersection},
      and (SLFV-2), 
     a 7th-order polynomial reconstruction in (SLFV-4), 
     and Gauss formulas in \cref{eq:Gauss-quadrature} with 
     $m_\xi,\ m_\eta \ge 4$ for pure cells and 
     $m_\xi,\ m_\eta \ge 12$ for interface cells.
   \item eighth-order convergent both in time and in space
     if \revise{we choose $h_L = O(h^2)$,} 
     an 8th-order RK method
      in \cref{eq:RK}, \cref{eq:RKfromIntersection},
      and (SLFV-2), 
     a 9th-order polynomial reconstruction in (SLFV-4), 
     and Gauss formulas in \cref{eq:Gauss-quadrature} with 
     $m_\xi,\ m_\eta \ge 5$ for pure cells and 
     $m_\xi,\ m_\eta \ge 15$ for interface cells.
    \end{itemize}
\end{theorem}
\begin{proof}
  The approximation of $\Omega$ with $\Omega_c$
   is fourth-, sixth-, and eighth-order accurate 
   when $h_L=O(h)$, $O\left(h^{\frac{3}{2}}\right)$, and $O(h^2)$, 
   respectively \cite{hu2025ARMS}. 
  Here we only give the proof for the fourth-order case, 
   as those for other cases are similar.

\revise{ 
  Define a max-norm of pointwise errors of the SLFV method in \cref{def:SLFV}
  as 
  \begin{equation}
    \label{eq:maxNormPointwise}
    e_n \coloneqq 
    \sup\nolimits_{\bx\in \Omega}\left|\tilde{\rho}^{n}(\bx)-\rho(\bx, t_{n})\right|.
  \end{equation}

  Suppose there exists a constant $C$ independent on $k$ such that, 
  \begin{equation}
    \label{eq:error_accumulation}
    \forall n\in \mathbb{N}, \quad 
    e_{n+1} \le (1+Ck)e_n + O(h^5+k^5).
  \end{equation}
  Then a straightforward induction
  and the identity $\lim_{n\rightarrow \infty}(1+\frac{C}{n})^n=\exp(C)$ yield 
  \begin{equation}
    \label{eq:interErrorEstimate}
    \begin{array}{l}
    e_{n} \le \exp(CT)e_0 + O\left(\frac{1}{k}h^5+k^4\right) 
    = O\left(\frac{1}{k}h^5+k^4\right),
    \end{array}
  \end{equation}
  where we have assumed that
  the initial error is $e_0=O\left(\frac{1}{k}h^5+k^4\right)$. 

  For any control volume $\bi$, the (cell-averaged) solution error satisfies
  \begin{equation*}
    \begin{split}
    \left|\avg{\rho}_{\bi}^n - \avg{\rho(\bx, t_n)}_{\bi}\right|
    =  &\left|I_{\bi}(\tilde{\rho}^n) - \avg{\rho(\bx, t_n)}_{\bi}\right|\\
    \le&\left|I_{\bi}(\tilde{\rho}^n) - I_{\bi}(\rho(\bx, t_n))\right|
       +\left|I_{\bi}(\rho(\bx, t_n)) - \avg{\rho(\bx, t_n)}_{\bi}\right|, 
    \end{split}
  \end{equation*}
  where the equality follows from (SLFV-5);
  in the second line, 
  the first RHS term is a linear combination
  of pointwise errors at the quadrature points in cell $\bi$ 
  and the second the truncation error of the quadrature formula.
  For any pure cell, 
  the quadrature is performed
  by recursively invoking one-dimensional Gauss formulas, yielding
  $\left|I_{\bi}(\rho(\bx, t_n))
    - \avg{\rho(\bx, t_n)}_{\bi}\right|=O(h^5)$, 
  which also holds for an interface cell
  because of \cref{lem:accuracyAnalysisQuadrilateral}
  and the fact of
  $q=3$ and $\kappa=4$ giving $\left\lceil\frac{(\kappa+2)q}{2}\right\rceil=9$.
  Then, $\sum\nolimits_{m=1}^{M}\omega_{\bi,m}=1$
  and \cref{eq:interErrorEstimate} imply 
  \begin{equation*}
    \begin{split}
    \left|\avg{\rho}_{\bi}^n - \avg{\rho(\bx, t_n)}_{\bi}\right|
    \le&\sum\nolimits_{m=1}^{M}\omega_{\bi,m}\left|
          \tilde{\rho}^n(\bx_{\bi,m}) - \rho(\bx_{\bi,m},t_n)
        \right| + O(h^5)\\
    \le& e_n\sum\nolimits_{m=1}^{M}\omega_{\bi,m} + O(h^5)
    =   O\left(\frac{1}{k}h^5+k^4\right)
    \end{split}
  \end{equation*}
  and the fourth-order convergence of the SLFV method
  follows from $k=O(h)$.

  We now prove the recurrence relation \cref{eq:error_accumulation}. 
  Denote by $\cT$ and $\cT_k$ the exact and discrete operator,
  respectively, for solving the ODE system
  \begin{equation}
    \label{eq:coupledODE}
    \frac{\dif}{\dif t}
    \begin{bmatrix}
      \rho \\ \mathbf{x} 
    \end{bmatrix}
    =
    \begin{bmatrix}
      S \\ \mathbf{u}
    \end{bmatrix}(\mathbf{x},t), 
  \end{equation}
  which is the combination of \cref{eq:lagrangian_derivative} and \cref{eq:ODE-flow-map}.
  Accordingly, denote by $\overleftarrow{\bx}^*\coloneqq \phi_{t_{n+1}}^{-k}(\bx)$
  and $\overleftarrow{\bx}$
  the exact and approximate preimage of $\bx$, respectively. 

  For any quadrature point $\bx$ with $\overleftarrow{\bx}\in\Omega$, 
  we adopt the shorthand notation
  \mbox{$\mathbf{v}\coloneqq\left[f(\overleftarrow{\bx}),
    \overleftarrow{\bx}\right]^\top$}, 
  $\mathbf{w}\coloneqq\left[\rho(\overleftarrow{\bx}^*, t_n), 
    \overleftarrow{\bx}^*\right]^\top$ 
  and deduce, from \cref{def:SLFV}, 
  \begin{equation}
    \label{eq:pointwiseErrorCase1}
    \begin{split}
      &\left|\tilde{\rho}^{n+1}(\bx) - \rho(\bx, t_{n+1})\right|
      = \left|\cT_k \mathbf{v}
        - \cT \mathbf{w} \right|
      \le\left|\cT_k\mathbf{v} - \cT \mathbf{v} \right|
      +\left|\cT \mathbf{v} - \cT \mathbf{w} \right|
      \\
      =&\left|\cT \mathbf{v} - \cT \mathbf{w} \right| + O(k^5)
      \le (1+Ck) \left\|\mathbf{v}-\mathbf{w} \right\| + O(k^5)
      \\
      =& (1+Ck) \left |
           f(\overleftarrow{\bx}) 
           - \rho(\overleftarrow{\bx}^*, t_n) 
         \right| + O(k^5)
      \le (1+Ck) e_n + O(h^5 + k^5), 
    \end{split}
  \end{equation}
   where $f$ is the local bi-variate polynomial 
   described in \Cref{sec:reconstruction}. 
   In \cref{eq:pointwiseErrorCase1},
   the first step follows from (SLFV-4), 
   the second from the triangular inequality, 
   the third and the fifth from the 4th-order accuracy of the RK method 
   in $\cT_k$,
   the fourth from \cref{lem:well-posed},
   and the last from
   \begin{equation*}
     \begin{split}
       \left|f(\overleftarrow{\bx}) 
         - \rho(\overleftarrow{\bx}^*, t_n)\right|
       \le&\left|f(\overleftarrow{\bx})
         - \tilde{\rho}^n(\overleftarrow{\bx})\right|
       + \left|\tilde{\rho}^n(\overleftarrow{\bx})
         - \rho(\overleftarrow{\bx}, t_n)\right|
       + \left|\rho(\overleftarrow{\bx}, t_n)
         - \rho(\overleftarrow{\bx}^*, t_n)\right|
       \\
       =& \left|\tilde{\rho}^n(\overleftarrow{\bx})
         - \rho(\overleftarrow{\bx}^*, t_n)\right| + O(h^5+k^5)
       \le e_n + O(h^5 + k^5),
     \end{split}
   \end{equation*}
  where the first inequality follows from the triangular inequality, 
  the equality from the 5th-order reconstruction in space
  and the 4th-order accuracy of time integration,
  and the last inequality from \cref{eq:maxNormPointwise}. 

  For a quadrature point $\bx$ with $\overleftarrow{\bx}\not\in\Omega$, 
   we calculate $\tilde{\rho}^{n+1}(\bx)$ by \cref{eq:RKfromIntersection}.
  As $k\rightarrow 0$, 
   the Newton iteration in \Cref{sec:penetration} must converge. 
  Denote by $\tilde{\bx}$ and $\tilde{t}$ 
  the calculated intersection point and time, respectively, 
  adopt the shorthand notation
  \mbox{$\mathbf{v}_{\times}\coloneqq\left[\rho_{\mathrm{bc}}(\tilde{\bx}, \tilde{t}), 
          \tilde{\bx}\right]^\top$}, 
  $\mathbf{w}_{\times}\coloneqq\left[\rho_{\mathrm{bc}}(\bx^*, t^* ), 
          \bx^*\right]^\top$, 
        and we have
  \begin{equation}
    \label{eq:pointwiseErrorCase2}
     \begin{split}
       &\left|\tilde{\rho}^{n+1}(\bx) - \rho(\bx, t_{n+1}) \right|
       =\left|
        \cT_k \mathbf{v}_{\times}
        - \cT \mathbf{w}_{\times}
      \right|
      \le \left|\cT_k \mathbf{v}_{\times}
        - \cT \mathbf{v}_{\times}
      \right|
      +\left|\cT \mathbf{v}_{\times}
        - \cT \mathbf{w}_{\times}
      \right|
      \\
      =&\left|\cT \mathbf{v}_{\times}
        - \cT \mathbf{w}_{\times} \right| + O(k^5)
      \le (1+Ck) \left\| \mathbf{v}_{\times}- \mathbf{w}_{\times}\right\| + O(k^5)
      = O(k^5), 
     \end{split}
   \end{equation}
   where the first step follows from \cref{eq:RKfromIntersection}, 
  and the second from the triangular inequality,
  the third from the 4th-order accuracy of time integration,
  the fourth from \cref{lem:well-posed},
  and the last from the smoothness of $\rho_{\mathrm{bc}}$
  and $\|[\tilde{\bx},\tilde{t}]^\top
    -[\bx^*, t^*]^\top\| = O(k^5)$.

  At any non-quadrature location $\bx$,
  the numerical result of SLFV is defined as
  a linear combination of those at quadrature points,
  with the weights depending only on the Gauss quadrature formula
  and the local bi-variate polynomial in (SLFV-4).
  Thus 
  $\left|\tilde{\rho}^{n+1}(\bx) - \rho(\bx, t_{n+1}) \right|$
  also satisfies \cref{eq:pointwiseErrorCase1}
  or \cref{eq:pointwiseErrorCase2}.
  Therefore \cref{eq:error_accumulation} holds. 
}
\end{proof}  

\revise{
  In the above proof, 
  the ODE system \cref{eq:coupledODE} numerically
  solved by the SLFV method is linear.
Nonetheless,
 the analysis based on the recurrence relation \cref{eq:error_accumulation}
 may also apply to some nonlinear problems
 such as the Vlasov-Poisson system; 
 see \cite{Besse2008ConvergenceOC}.
}

The analysis in this section answers (Q-\ref{Q-4}).

\section{Tests}
\label{sec:numerical-tests}
In this section, 
 we test the proposed SLFV method 
 with a wide range of benchmark problems
 to demonstrate its high-order accuracy and good conditioning,
 its effectiveness for incoming penetration conditions, 
 and its capability of handling irregular domains 
 with arbitrarily complex topology and geometry. 
In \Cref{table:variousConditionsForTests},
 we summarize various configurations of these tests, 
 in which we employ the classical fourth-order RK method, 
 Verner's sixth-order RK method \cite{verner1978explicit} 
 and Dormand and Prince's eighth-order RK method
 \cite{prince1981high}
 to test the fourth-, sixth-, and eighth-order SLFV methods,
 respectively.

\begin{table}  
  \centering
   \caption{Overview of test cases in \Cref{sec:numerical-tests}. 
    The benchmark problems in the first two subsections
    are meant to check convergence rates
    of the proposed SLFV method on regular or irregular domains,
    for solenoidal or non-solenoidal velocity fields,
    and with penetration or no-penetration conditions.
    The test in \Cref{sec:penetr-veloc-bound-1}
     focuses on the accuracy comparison of the SLFV method
     with the EL-RK-FV method
     by Nakao, Chen, and Qiu \cite{nakao2022eulerian},
    those in \Cref{sec:penetr-veloc-bound-4}
    demonstrate the capability of SLFV
    in tackling practical applications
    such as chaotic mixing,
    and those in \Cref{sec:penetr-veloc-bound-5}
    confirm the applicability of SLFV
    to domains with complex topology and geometry.
  }
  \label{table:variousConditionsForTests}
  \begin{tabular}{c|c|c|c|c}
    \hline
    Subsection & 
    \makecell{zero source \\ term} & 
    \makecell{regular \\ domain} & 
    \makecell{incoming \\ penetration} & 
    \makecell{solenoidal \\ velocity}\\
    \hline 
\ref{sec:no-penetr-veloc} & yes & yes & no & yes   \\
    \hline
\ref{sec:penetr-veloc-bound-2} &  no  & no  & yes & no \\
    \hline
\ref{sec:penetr-veloc-bound-1} & yes & yes & yes &  yes \\
    \hline
\ref{sec:penetr-veloc-bound-4} &  yes  & yes & yes  & yes \\
    \hline
\ref{sec:penetr-veloc-bound-5} &  yes  &  no &  yes  & yes \\
    \hline        
  \end{tabular}
\end{table}

The $L^p$ norm of a scalar function
 $g: \Omega\to \mathbb{R}$ is defined as
\begin{equation*}
  \|g\|_p = 
  \left\{
    \begin{aligned}
      &\left( 
        \sum\limits_{\mC_\bi\subset \Omega} 
        \|\mC_{\bi}\| \cdot 
        \left|\avg{g}_\bi\right|^p 
      \right)^{\frac{1}{p}} \quad 
      && \mathrm{if}\ p=1,2; \\
      & \max\limits_{\mC_\bi\subset \Omega} |\avg{g}_\bi| \quad 
      && \mathrm{if}\ p=\infty, 
    \end{aligned}
  \right.
\end{equation*}
where the cut cell $\mC_\bi$ and the cell average $\avg{g}_\bi$
are given in \cref{eq:cutCellAndBdry}
and \cref{eq:cellAndFaceAverage}, respectively. 

Denote by $\rho_h$ the numerical solution on a grid of size $h$
 and $r$ the grid refinement ratio. 
If the analytic solution $\rho$ is available,
 the error of $\rho_h$
 is measured by $E_h \coloneqq \rho_{h} - \rho$; 
 otherwise Richardson extrapolation is employed, 
 i.e., $E_h \coloneqq \rho_{h} - \rho_{\frac{h}{r}}$.
Then the convergence rate is given by
 $\kappa \coloneqq \log_r \frac{\|E_{rh}\|_p}{\|E_{h}\|_p}$.

\subsection{A solenoidal velocity
  field on a regular no-penetration domain} 
\label{sec:no-penetr-veloc}
On the unit box $\Omega = (0, 1)^2$,  
we set $S=0$ and 
\begin{equation*}
  \bu(x,y,t) = 
  \begin{pmatrix}
    -\sin^2(\pi x)\sin(\pi y)\cos(\pi y)\\
    \sin(\pi x)\cos(\pi x) \sin^2(\pi y)
  \end{pmatrix}\sin(t)
\end{equation*}
in \cref{eq:advection_control_equation}
and choose the initial and boundary conditions of $\rho$ to be
\begin{equation*}
  \rho_{\text{init}}(x,y) = 
  \rho_{\text{bc}}(x, y,t) = 
  x+y.
\end{equation*}

We solve the advection equation \cref{eq:advection_control_equation}
 to the final time $T=1$ with the time step size as $k = 8h$. 
Since no analytic solution is available,
we calculate errors and convergence rates of SLFV
 via Richardson extrapolation
 and present them in 
 \cref{table:regularNopenetrationIncompressibleZeroSourceTerm},
 which clearly demonstrate
 the fourth- and sixth-order convergence rates of SLFV. 

 \begin{table}
  \centering
  \caption{Errors and convergence rates of 
    the $\kappa$-th order SLFV method 
    for the test in \Cref{sec:no-penetr-veloc}.
    The errors are calculated by Richardson extrapolation.
  }
  \label{table:regularNopenetrationIncompressibleZeroSourceTerm}
  \begin{tabular}{c|ccccccccc}
    \hline
    $\kappa$ && $\frac{1}{32}$--$\frac{1}{64}$ & rate & 
                $\frac{1}{64}$--$\frac{1}{128}$ & rate & 
                $\frac{1}{128}$--$\frac{1}{256}$ & rate & 
                $\frac{1}{256}$--$\frac{1}{512}$\\
    \hline\hline
    \multirow{3}{*}{4} & $L^\infty$ & 
    1.16e-04 & 3.98 &7.36e-06 & 4.01 & 4.57e-07 & 4.09 & 2.68e-08 \\
    ~ & $L^1$ & 
    2.95e-05 & 4.01 & 1.83e-06 & 4.02 & 1.13e-07 & 4.09 & 6.63e-09 \\
    ~ & $L^2$ & 
    4.65e-05 & 4.01 & 2.89e-06 & 4.01 & 1.79e-07 & 4.09 & 1.05e-08\\
    \hline\hline
    \multirow{3}{*}{6} & $L^\infty$ & 
    1.44e-07 & 5.68 & 2.81e-09 & 5.77 & 5.15e-11 & 5.93 & 8.43e-13 \\
    ~ & $L^1$ & 
    3.08e-08 & 5.84 & 5.37e-10 & 5.99 & 8.42e-12 & 6.04 & 1.28e-13 \\
    ~ & $L^2$ & 
    4.65e-08 & 5.83 & 8.19e-10 & 5.91 & 1.36e-11 & 5.98 & 2.16e-13 \\
    \hline        
  \end{tabular}
\end{table}

\revise{
 Due to $S=0$ and $\nabla\cdot\mathbf{u}=0$,
 \cref{eq:advection_control_equation} 
 is equivalent to the scalar conservation law
 $\frac{\partial \rho}{\partial t} + \nabla\cdot(\rho\mathbf{u})=0$. 
We define the relative error of mass conservation as
\begin{equation}
  \label{eq:relation-mass-error}
  E_{\textnormal{mass}}^{\textnormal{rel}} (t_n) \coloneqq 
  \frac{\iint_{\Omega} \rho(x, y, t_n) \ \dif x \dif y 
    - \iint_{\Omega} \rho_{\textnormal{init}}(x, y) \ \dif x \dif y}{
    \iint_{\Omega} \rho_{\textnormal{init}}(x, y) \ \dif x \dif y}.
\end{equation}
The evolution of $E_{\textnormal{mass}}^{\textnormal{rel}}$
 for the fourth-order SLFV method 
 with $h = \frac{1}{256}$ is shown in \cref{fig:massErr_t40}, 
 where 
 $\max_{t_n}\left|E_{\textnormal{mass}}^{\textnormal{rel}}(t_n)\right|$
 is about $5\times 10^{-11}$,
 which is much smaller than $4.57\times 10^{-7}$,
 the $L^{\infty}$-norm of the solution error in 
 \cref{table:regularNopenetrationIncompressibleZeroSourceTerm}.
 Since the error of mass conservation is always
 smaller than the solution error,
 a high-order method tends to accurately conserve the total mass.
}

\begin{figure}
  \centering
  \includegraphics[width=0.65\textwidth]{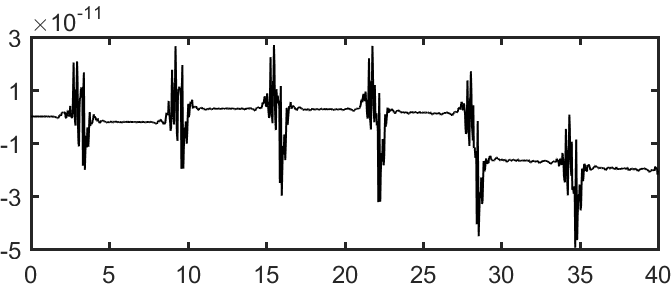}
  \vspace{-11pt}
  \caption{\revise{The evolution of the relative error of mass conservation
      for the fourth-order SLFV method
      in solving the test in \Cref{sec:no-penetr-veloc}
      on a grid of $h=\frac{1}{256}$
      within the time interval $[0, 40]$. 
    The horizontal and vertical axes represent the time
    and the relative error in \cref{eq:relation-mass-error},
    respectively.
    }}
  \label{fig:massErr_t40}
\end{figure}

\subsection{A non-solenoidal velocity field
  on an irregular penetration domain}
\label{sec:penetr-veloc-bound-2}

In this test we set $\Omega$ to a trapezoid 
 with its vertices at $(0.5, 0)$, $(2,0)$, $(2,2)$, and $(0,2)$. 
The velocity and the exact solution of $\rho$ are 
\begin{equation*}
  \bu(x, y, t) = 
  \begin{pmatrix}
    x^2 + y^2 + t\\
      x + y -2t
  \end{pmatrix}; \quad 
  \rho(x, y, t) = \sin(x + y +t), 
\end{equation*}
 for which the advection equation
 \cref{eq:advection_control_equation}
 is closed by the source term
\begin{equation*}
  S = \cos(x+y+t)(1+x^2+y^2+x+y-t).
\end{equation*}

We solve \cref{eq:advection_control_equation}
 to the final time $T=1$ with $k = 8h$. 
Errors and convergence rates are presented in 
\cref{table:trapezoidalPenetrationCompressibleNonZeroSourceTerm}, 
 confirming the fourth- and sixth-order convergence rates
 of SLFV over the irregular domain. 

 \begin{table}
  \centering
  \caption{Errors and convergence rates of 
    the $\kappa$-th order SLFV method 
    for the test in \Cref{sec:penetr-veloc-bound-2}.}  
  \label{table:trapezoidalPenetrationCompressibleNonZeroSourceTerm}
  \begin{tabular}{c|ccccccccc}
    \hline
    $\kappa$ && $h=\frac{1}{64}$ & rate & $h=\frac{1}{128}$ & rate & 
                $h=\frac{1}{256}$ & rate & $h=\frac{1}{512}$\\
    \hline\hline
    \multirow{3}{*}{4} & $L^\infty$ & 
    1.07e-03 & 4.14 & 6.03e-05 & 4.03 & 3.68e-06 & 4.01 & 2.29e-07 \\
    ~ & $L^1$ & 
    7.21e-04 & 4.05 & 4.35e-05 & 4.02 & 2.68e-06 & 4.01 & 1.67e-07 \\
    ~ & $L^2$ & 
    6.55e-04 & 4.10 & 3.81e-05 & 4.03 & 2.33e-06 & 4.01 & 1.45e-07 \\
    \hline\hline
    \multirow{3}{*}{6} & $L^\infty$ & 
    1.79e-05 & 5.45 & 4.08e-07 & 5.86 & 7.01e-09 & 5.96 & 1.13e-10 \\
    ~ & $L^1$ & 
    2.34e-06 & 5.68 & 4.57e-08 & 5.90 & 7.68e-10 & 5.97 & 1.22e-11 \\
    ~ & $L^2$ & 
    2.91e-06 & 5.43 & 6.74e-08 & 5.84 & 1.17e-09 & 5.96 & 1.89e-11 \\
    \hline         
  \end{tabular}
\end{table}

\subsection{Solid-body rotation 
  on a regular periodic penetration domain}
\label{sec:penetr-veloc-bound-1}
On the domain $\Omega = (-\pi, \pi)^2$, 
 we set the velocity and the exact solution of $\rho$ to be
\begin{equation*}  
    \bu(x, y, t) = 
    \begin{pmatrix}
      -y \\
      x
    \end{pmatrix}; \quad 
    \rho(x,y,t) = \exp(-3(x^2+y^2)), 
\end{equation*}
 for which the equation \cref{eq:advection_control_equation}
 is closed with the zero source term $S=0$. 
The boundary conditions of $\rho$ are set to be periodic
 and the final time is $T = 0.5$. 
This test is exactly the same as
 that in \cite[Example 4.5]{nakao2022eulerian}
 so that it is meaningful to compare our SLFV method
 to the EL-RK-FV method by Nakao, Chen, and Qiu
 \cite{nakao2022eulerian}.

\begin{table}
  \centering
  \caption{Errors and convergence rates of 
    the $\kappa$-th order SLFV method
    and the EL-RK-FV method
    by Nakao, Chen, and Qiu \cite{nakao2022eulerian}
    for the test in \Cref{sec:penetr-veloc-bound-1}
    with $\text{CFL} = \frac{k(\max|u| + \max |v|)}{h}$.
  }  
  \label{table:regularPenetrationIncompressibleZeroSourceTermRigigRotationSmallCFL}  
  \begin{tabular}{c|ccccccccc}
    \hline
    CFL &&  $h=\frac{1}{100}$ & rate & $h=\frac{1}{200}$ & rate & 
            $h=\frac{1}{300}$ & rate & $h=\frac{1}{400}$\\ 
    \hline\hline
    \multirow{16}{*}{0.95} & 
    \multicolumn{8}{c}{The EL-RK-FV method
        by Nakao, Chen, and Qiu \cite{nakao2022eulerian}} \\
    \cline{2-9}
    ~ & $L^\infty$ & 
    8.95e-05 & 6.85 & 7.74e-07 & 5.00 & 1.02e-07 & 4.99 & 2.43e-08 \\
    ~ & $L^1$ & 
    1.05e-04 & 6.60 & 1.08e-06 & 5.06 & 1.39e-07 & 4.99 & 3.32e-08 \\
    ~ & $L^2$ & 
    4.66e-05 & 6.25 & 6.13e-07 & 4.99 & 8.12e-08 & 4.99 & 1.93e-08 \\
    \cline{2-9}
    ~ & \multicolumn{8}{c}{Our fourth-order SLFV method} \\    
    \cline{2-9}
    ~ & $L^\infty$ & 
    3.55e-05 & 4.58 & 1.49e-06 & 4.15 & 2.76e-07 & 5.07 & 6.42e-08 \\
    ~ & $L^1$ & 
    1.02e-06 & 4.85 & 3.55e-08 & 4.87 & 4.92e-09 & 4.49 & 1.35e-09 \\
    ~ & $L^2$ & 
    3.46e-06 & 4.79 & 1.25e-07 & 4.91 & 1.71e-08 & 4.57 & 4.58e-09 \\
    \cline{2-9} 
    ~ & \multicolumn{8}{c}{Our sixth-order SLFV method} \\
    \cline{2-9}
    ~ & $L^\infty$ & 
    1.83e-06 & 6.19 & 2.50e-08 & 7.42 & 1.24e-09 & 6.60 & 1.85e-10 \\
    ~ & $L^1$ & 
    3.40e-08 & 6.63 & 3.43e-10 & 7.03 & 1.99e-11 & 6.36 & 3.18e-12 \\
    ~ & $L^2$ & 
    1.22e-07 & 6.54 & 1.31e-09 & 7.28 & 6.83e-11 & 6.29 & 1.12e-11 \\
    \cline{2-9}
    ~ & \multicolumn{8}{c}{Our eighth-order SLFV method} \\
    \cline{2-9}
    ~ & $L^\infty$ & 
    5.60e-08 & 8.50 & 1.55e-10 & 8.65 & 4.63e-12 & 7.95 & 4.70e-13 \\
    ~ & $L^1$ & 
    1.32e-09 & 8.85 & 2.86e-12 & 9.01 & 7.39e-14 & 7.14 & 9.47e-15 \\
    ~ & $L^2$ & 
    4.61e-09 & 8.77 & 1.06e-11 & 8.92 & 2.84e-13 & 7.19 & 3.59e-14 \\
    \hline\hline
    \multirow{16}{*}{8} & 
    \multicolumn{8}{c}{The EL-RK-FV method
        by Nakao, Chen, and Qiu \cite{nakao2022eulerian}} \\
    \cline{2-9}
    ~ & $L^\infty$ & 
    6.82e-05 & 7.16 & 4.77e-07 & 5.03 & 6.19e-08 & 5.00 & 1.47e-08 \\
    ~ & $L^1$ & 
    4.95e-05 & 6.68 & 4.84e-07 & 5.08 & 6.17e-08 & 4.99 & 1.47e-08 \\
    ~ & $L^2$ & 
    2.46e-05 & 6.39 & 2.92e-07 & 4.99 & 3.86e-08 & 4.88 & 9.22e-09 \\
    \cline{2-9} 
    ~ & \multicolumn{8}{c}{Our fourth-order SLFV method} \\    
    \cline{2-9}
    ~ & $L^\infty$ & 
    3.03e-05 & 4.85 & 1.05e-06 & 5.05 & 1.36e-07 & 5.13 & 3.10e-08 \\
    ~ & $L^1$ & 
    4.79e-07 & 4.72 & 1.81e-08 & 4.88 & 2.51e-09 & 5.31 & 5.44e-10 \\
    ~ & $L^2$ & 
    2.03e-06 & 4.77 & 7.42e-08 & 4.87 & 1.03e-08 & 5.23 & 2.29e-09 \\
    \cline{2-9}
    ~ & \multicolumn{8}{c}{Our sixth-order SLFV method} \\
    \cline{2-9}
    ~ & $L^\infty$ & 
    8.66e-07 & 6.88 & 7.33e-09 & 6.93 & 4.41e-10 & 6.82 & 6.20e-11 \\
    ~ & $L^1$ & 
    1.89e-08 & 7.02 & 1.46e-10 & 6.90 & 8.89e-12 & 6.78 & 1.26e-12 \\
    ~ & $L^2$ & 
    7.77e-08 & 6.96 & 6.23e-10 & 6.97 & 3.69e-11 & 6.81 & 5.19e-12 \\
    \cline{2-9} 
    ~ & \multicolumn{8}{c}{Our eighth-order SLFV method} \\
    \cline{2-9}
    ~ & $L^\infty$ & 
    4.67e-08 & 8.75 & 1.09e-10 & 8.85 & 3.00e-12 & 8.31 & 2.75e-13 \\
    ~ & $L^1$ & 
    7.59e-10 & 8.96 & 1.53e-12 & 8.51 & 4.84e-14 & 7.11 & 6.26e-15 \\
    ~ & $L^2$ & 
    3.20e-09 & 8.90 & 6.71e-12 & 8.71 & 1.96e-13 & 8.22 & 1.85e-14 \\
    \hline
  \end{tabular}
\end{table}

Errors and convergence rates of these two methods
 are presented in 
 \cref{table:regularPenetrationIncompressibleZeroSourceTermRigigRotationSmallCFL}, 
 where the fourth-, sixth-, and eighth-order convergence rates of SLFV
 are demonstrated
 for small and large CFL numbers.
While errors of the fourth-order SLFV are close
 to those of EL-RK-FV,
 those of the sixth-order and eighth-order SLFV methods
 are much smaller,
 confirming the advantage of higher-order SLFV methods.
For example,
 on the finest grid, the eighth-order SLFV
 is more accurate than EL-RK-FV
 by five orders of magnitude.
The $L^1$-norm being close to machine precision
 also indicates an excellent conditioning of SLFV.
 
It appears that the extension of EL-RK-FV to the sixth and eighth orders 
 involves splitting techniques that are more complex
 than that of the fourth-order EL-RK-FV. 
In contrast, 
 the extension of the fourth-order SLFV
 to higher orders is as simple as changing values of the input
 parameters; 
 see \cref{thm:accuracyOfSLFV} and its proof.

\begin{table}
  \centering
  \caption{\revise{CPU time (in seconds) of the $\kappa$-th order SLFV method for
    solving the test problem in \Cref{sec:penetr-veloc-bound-1}
    with $\text{CFL} = \frac{k(\max|u| + \max |v|)}{h}$ on an Intel
    Core i7-7500U CPU @ 2.70GHz.}}
  \label{table:SLFVCPU}
  \begin{tabular}{c|c|ccccccccccc}
    \hline
    $\text{CFL}$ & $\kappa$ & $h=\frac{1}{100}$ & rate &
                   $h=\frac{1}{200}$  & rate &
                   $h=\frac{1}{300}$  & rate &
                   $h=\frac{1}{400}$ \\ 
    \hline\hline
    \multirow{3}{*}{0.95} 
    ~ & 4 & 16.77 & 3.00 & 133.96 & 3.00 & 455.94 & 2.98 & 1075.50 \\
    ~ & 6 & 81.14 & 2.99 & 646.28 & 3.00 & 2176.83 & 2.95 & 5091.73 \\ 
    ~ & 8 & 118.39 & 3.01 & 955.75 & 3.00 & 3227.25 & 2.98 & 7600.98 \\
    \hline\hline
    \multirow{3}{*}{8} 
    ~ & 4 & 2.04 & 2.99 & 16.19 & 2.99 & 54.42 & 2.96 & 127.42 \\
    ~ & 6 & 9.66 & 2.99 & 76.72 & 3.00 & 259.04 & 2.97 & 608.44 \\
    ~ & 8 &14.89 & 3.01 & 119.56 & 3.03 & 408.68 & 2.95 & 955.20 \\
    \hline
  \end{tabular}
\end{table}

\revise{
  The CPU time of the SLFV method
  for generating results in 
  \cref{table:regularPenetrationIncompressibleZeroSourceTermRigigRotationSmallCFL}
  is reported in \cref{table:SLFVCPU}, 
  where the increase rate of CPU time is very close to 3, 
  demonstrating the \emph{optimal} complexity of SLFV
  that the CPU time is linearly proportional
  to the number of control volumes at each time step.
  In addition, these rates also imply that
  the expense of complex treatments
  on an $O\left(\frac{1}{h}\right)$ number of cells near the domain boundary
  is dominated by that of simple algorithmic steps
  on an $O\left(\frac{1}{h^2}\right)$ number of cells away from the boundary.
}

\subsection{Transient chaotic mixing
  on a regular penetration domain}
\label{sec:penetr-veloc-bound-4}

\begin{table}
  \centering
  \caption{Errors and convergence rates of 
    the $\kappa$-th order SLFV method 
    for the chaotic mixing test in \Cref{sec:penetr-veloc-bound-4}
    with $T=0.1$ and $k=0.4h$.
    The errors are calculated by Richardson extrapolation.
  }
  \label{table:regularPenetrationIncompressibleZeroSourceTermLargeVortex}
  \begin{tabular}{c|ccccccccc}
    \hline
    $\kappa$ && $\frac{1}{32}$--$\frac{1}{64}$ & rate & 
                $\frac{1}{64}$--$\frac{1}{128}$ & rate & 
                $\frac{1}{128}$--$\frac{1}{256}$ & rate & 
                $\frac{1}{256}$--$\frac{1}{512}$ \\ 
    \hline\hline
    \multirow{3}{*}{4} & $L^\infty$ & 
    3.18e-04 & 3.53 & 2.75e-05 & 3.68 & 2.14e-06 & 3.94 & 1.40e-07 \\
    ~ & $L^1$ & 
    5.24e-05 & 4.01 & 3.26e-06 & 3.90 & 2.18e-07 & 3.97 & 1.39e-08 \\
    ~ & $L^2$ & 
    7.92e-05 & 3.93 & 5.20e-06 & 3.90 & 3.48e-07 & 3.95 & 2.24e-08 \\
    \hline\hline
    \multirow{3}{*}{6} & $L^\infty$ & 
    6.25e-05 & 5.68 & 1.22e-06 & 5.65 & 2.42e-08 & 5.93 & 3.97e-10 \\
    ~ & $L^1$ & 
    9.09e-06 & 5.90 & 1.52e-07 & 6.01 & 2.36e-09 & 5.74 & 4.42e-11 \\
    ~ & $L^2$ & 
    1.42e-05 & 5.86 & 2.43e-07 & 5.97 & 3.89e-09 & 5.87 & 6.67e-11 \\
    \hline
  \end{tabular}
\end{table}

One realworld phenomenon that is 
 challenging to numerical methods for the advection equation
 is the transient chaotic fluid mixing problem,
 of which the solution may evolve to be highly complex 
 even for simple initial conditions and smooth velocity fields.
Such an example is studied 
 in \cite{mathew2007optimal},
 where the domain is a unit box $\Omega = (0,1)^2$,
 the source term in \cref{eq:advection_control_equation}
 is set to $S = 0$,
 the initial condition is $\rho_{\text{init}}(x, y) = \sin(2 \pi y)$, 
 the boundary condition of $\rho$ is periodic, 
 and the velocity field\footnote{
  In \cite{mathew2007optimal}, 
  the second component of $\mathbf{u}_1$ is 
  $-\cos(2 \pi x) \sin(2 \pi y)$,
  of which the minus sign appears to be incorrect
  since the solutions of SLFV with \cref{eq:chaoticMixingVel}
  agree very well with figures in \cite{mathew2007optimal}
  but those with the incorrect minus sign
  differ largely with results in \cite{mathew2007optimal}.
}  is 
\begin{equation}
  \label{eq:chaoticMixingVel}
  \left\{
  \begin{aligned}
  \bu(x,y,t) &= \alpha_1(t) \bu_1(x,y) + \alpha_2(t) \bu_2(x,y), \\
  \bu_1(x,y) &= 
  \begin{pmatrix}
    -\sin (2 \pi x) \cos(2 \pi y)\\
    \cos(2 \pi x) \sin(2 \pi y)
  \end{pmatrix}, \\ 
  \bu_2(x,y) &= 
  \begin{pmatrix}
    -\sin (2 \pi (x-0.25)) \cos(2 \pi (y-0.25))\\
    \cos(2 \pi (x-0.25)) \sin(2 \pi (y-0.25))
  \end{pmatrix}, 
  \end{aligned}\right.
\end{equation}
where the coefficients $\alpha_1(t)$ and $\alpha_2(t)$ 
are cubic splines fitted from characteristic points
of the curves in \cite[Figure 2(b)]{mathew2007optimal}.

Errors and convergence rates of SLFV at $T=0.1$ with $k=0.4h$
 are listed in 
 \Cref{table:regularPenetrationIncompressibleZeroSourceTermLargeVortex}, 
 clearly demonstrating the fourth- and sixth-order convergences rates. 
In \Cref{fig:densityEvolutionoptimalControlOfMixing},
 we present snapshots of the solution at six time instances,  
 of which the patterns and structures
 are visually indistinguishable 
 from those in \cite[Figure 4]{mathew2007optimal}.
 
\begin{figure}[p]  
  \centering
  \subfigure[$t=0$]{
    \includegraphics[width=0.3\textwidth, height=35mm]{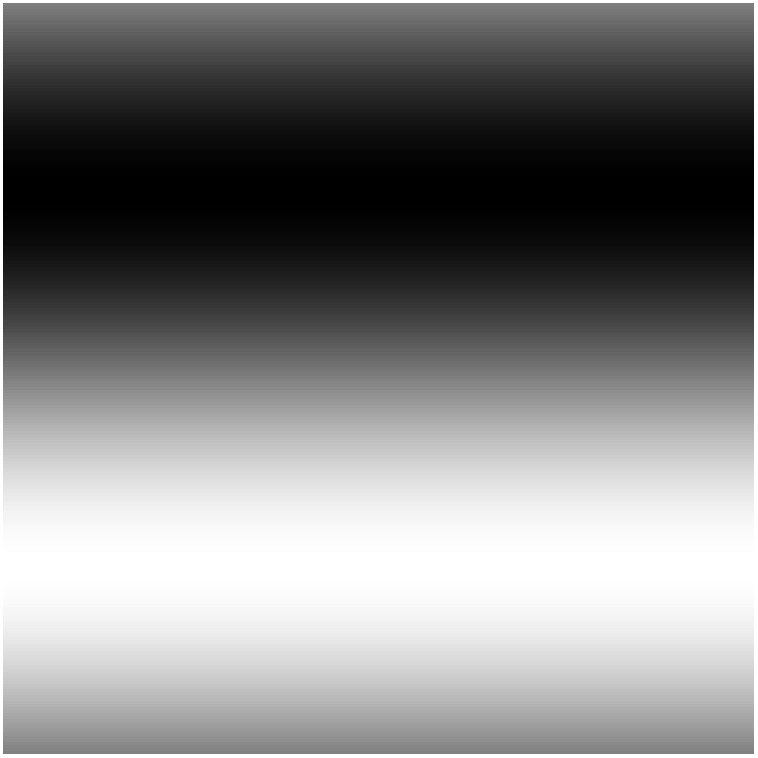}
  }
  \subfigure[$t=0.2$]{
    \includegraphics[width=0.3\textwidth, height=35mm]{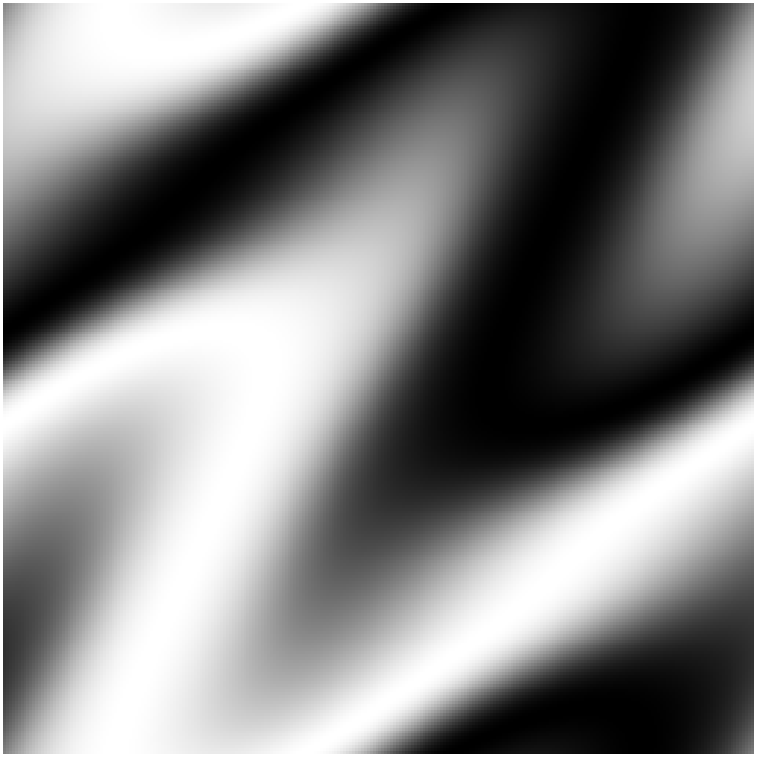}
  }
  \subfigure[$t=0.4$]{
    \includegraphics[width=0.3\textwidth, height=35mm]{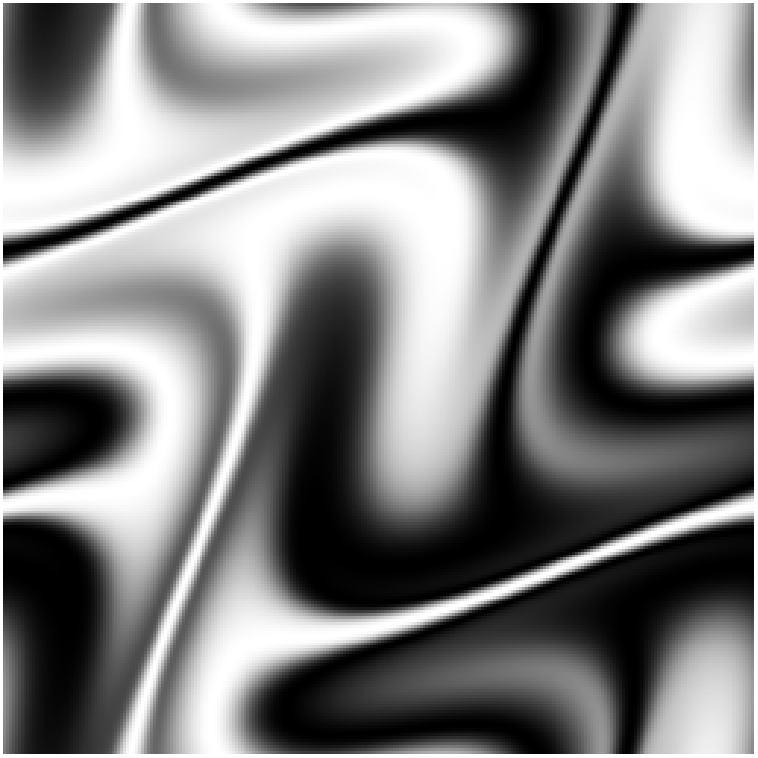}
  } \\
  ~
  \subfigure[$t=0.6$]{
    \includegraphics[width=0.3\textwidth, height=35mm]{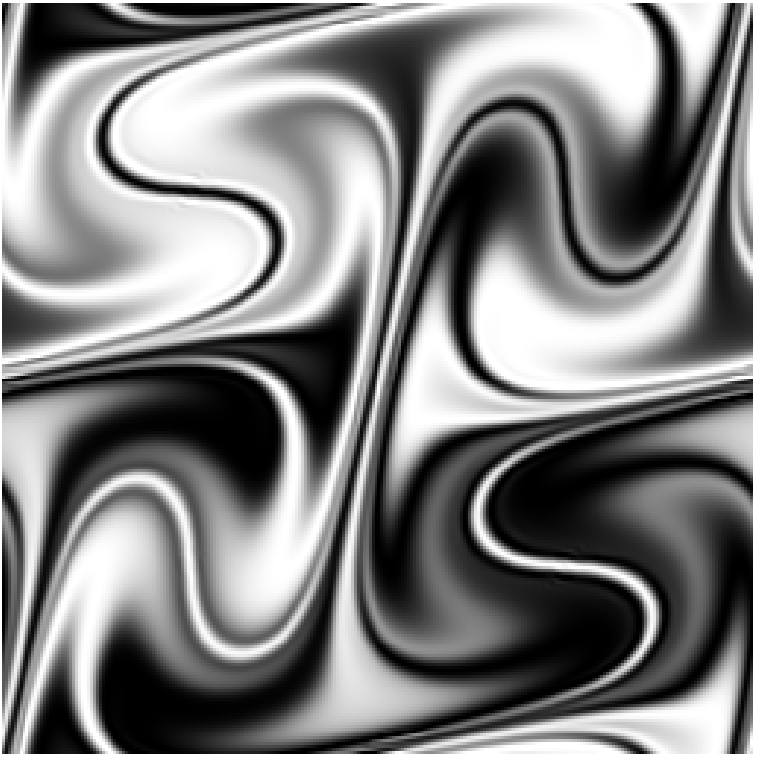}
  }
  \subfigure[$t=0.8$]{
    \includegraphics[width=0.3\textwidth, height=35mm]{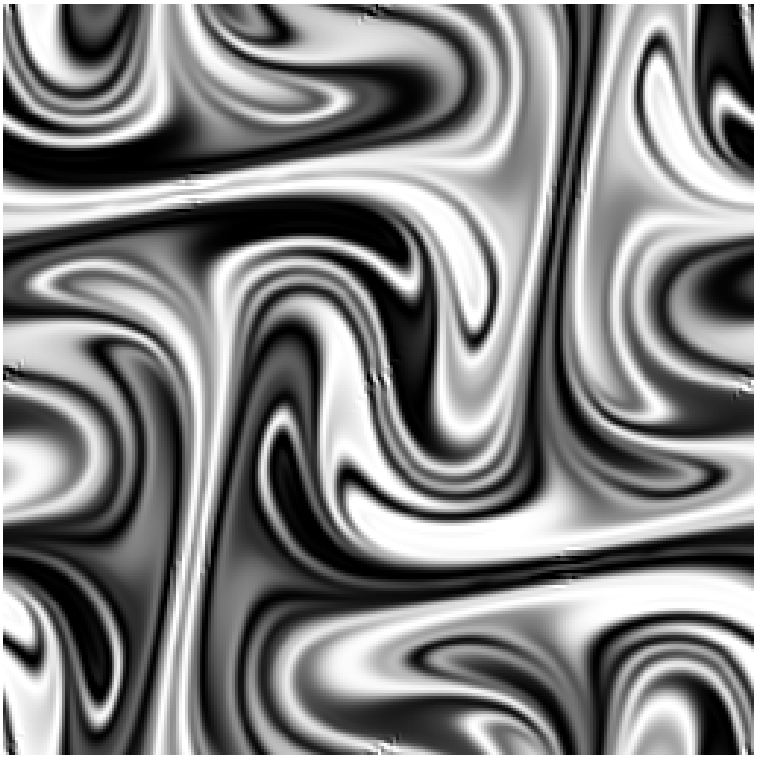}
  }
  \subfigure[$t=1$]{
    \includegraphics[width=0.3\textwidth, height=35mm]{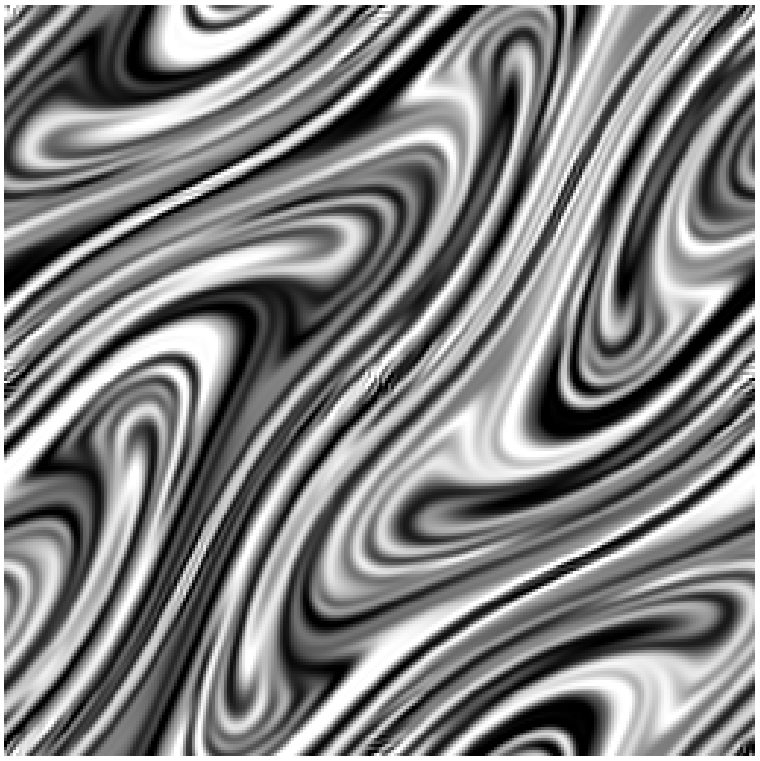}
  }
  \caption{Results of the SLFV method for the chaotic mixing test
    in \Cref{sec:penetr-veloc-bound-4}
    on a regular domain with $h = \frac{1}{256}$. 
    The lightest color corresponds to the value of $1$ 
    and the darkest color corresponds to the value of $-1$.} 
  \label{fig:densityEvolutionoptimalControlOfMixing}
\end{figure}

\subsection{Transient chaotic mixing
  on a complex penetration domain}
\label{sec:penetr-veloc-bound-5}

The test in \Cref{sec:penetr-veloc-bound-4} is
 made even more challenging by changing the domain $\Omega$
 from the unit box
 to a panda adapted from \cite[Figure 10]{zhang2020boolean};
 see \cref{fig:densityEvolutionoptimalControlOfMixingOnPanda}.
The complex topology and geometry present
 serious challenges for Boolean algorithms
 in cutting cells with the irregular boundary $\partial\Omega$
 and for the Newton iteration in intersecting pathlines
 with the domain boundary.
 
While the velocity on $\partial\Omega$
 is still given by \cref{eq:chaoticMixingVel}, 
 the boundary condition of $\rho$ 
 is obtained from the computational results
 on the regular unit box with the same grid size $h$
 via a sufficiently accurate polynomial interpolation.
All other parameters of this test
 are exactly the same as those in \Cref{sec:penetr-veloc-bound-4}.
Consequently, we expect that solutions on the panda
 match very well with those on the unit box,
 as if solutions on the panda were ``cut''
 from those on the regular domain by the shape of the panda.
 
Note that values of $\rho$ are not needed 
 for all points on $\partial\Omega$,
 since the steps of SLFV in \cref{def:SLFV}
 only concern those at the intersections of pathlines to the boundary of panda.
This highlights an advantage of SLFV
 as being worry free about the type 
 of the boundary condition \cref{eq:advection_control_equation_bc}.
 
\begin{figure}[p]
  \centering
  \subfigure[$t=0$]{
    \label{fig:densityEvolutionoptimalControlOfMixingOnPanda_a}
    \includegraphics[width=0.3\textwidth, height=35mm]
    {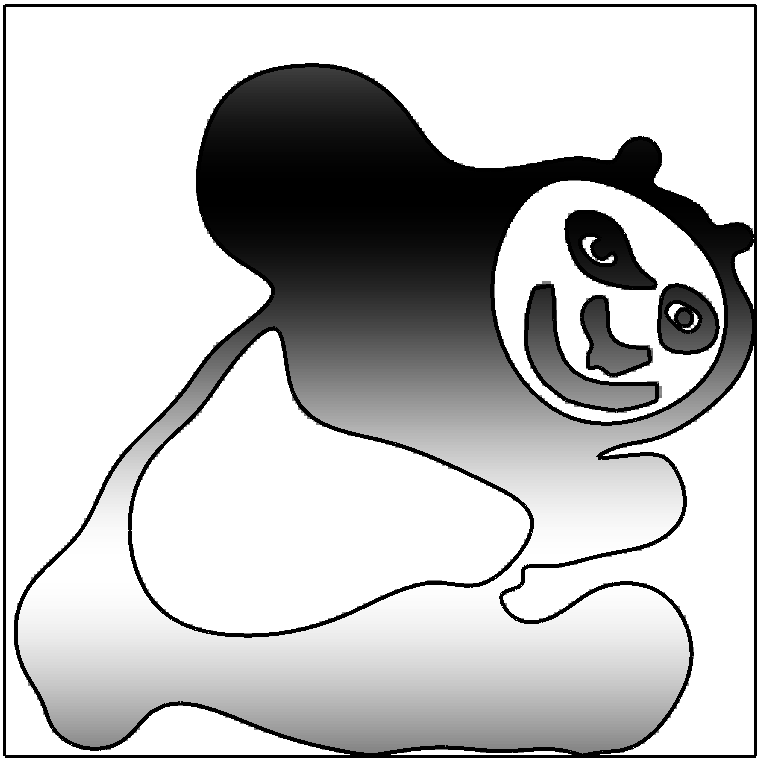}
  }
  \subfigure[$t=0.2$]{
    \includegraphics[width=0.3\textwidth, height=35mm]
    {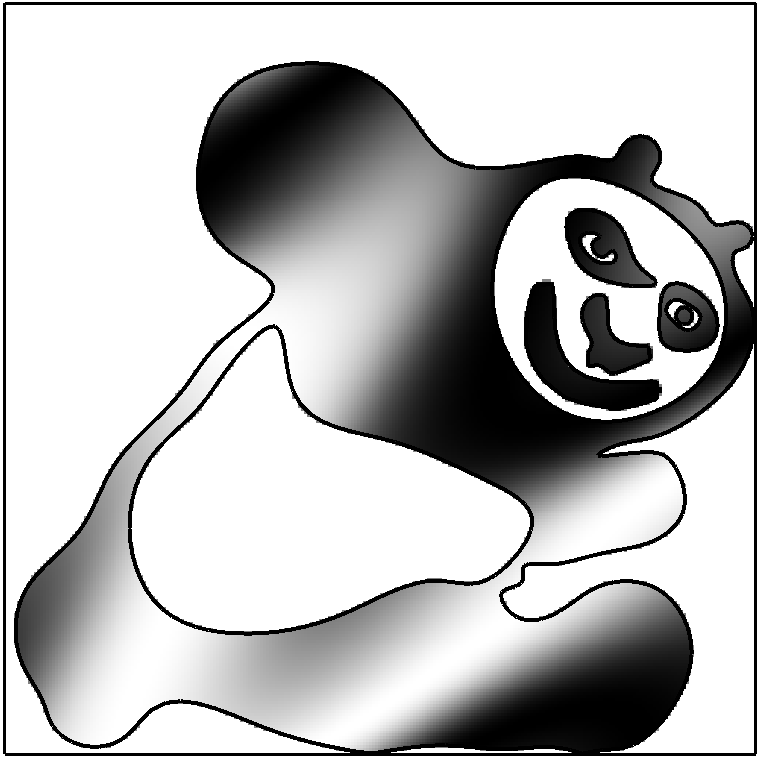}
  }
  \subfigure[$t=0.4$]{
    \includegraphics[width=0.3\textwidth, height=35mm]
    {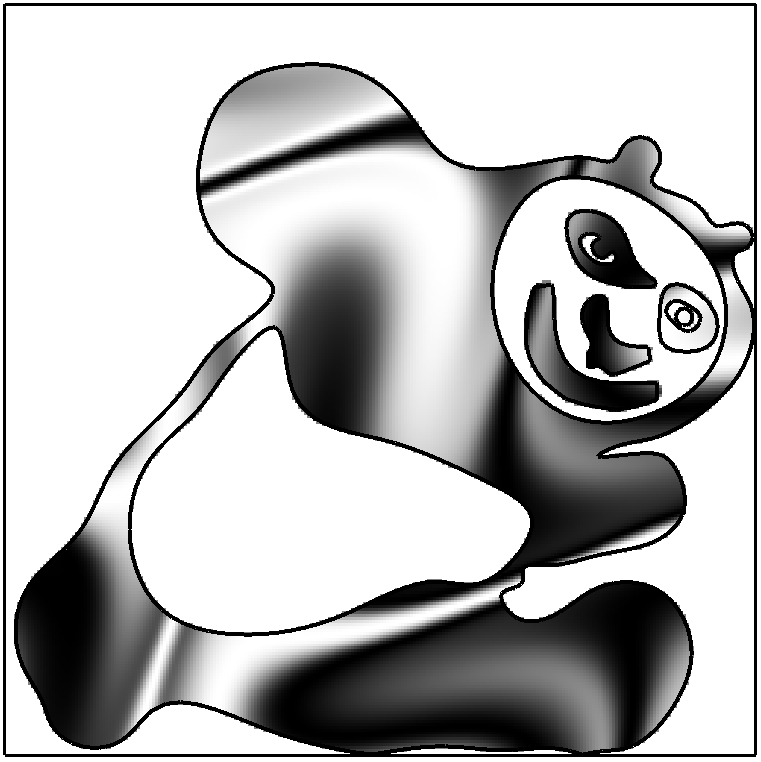}
  } \\
  ~
  \subfigure[$t=0.6$]{
    \includegraphics[width=0.3\textwidth, height=35mm]
    {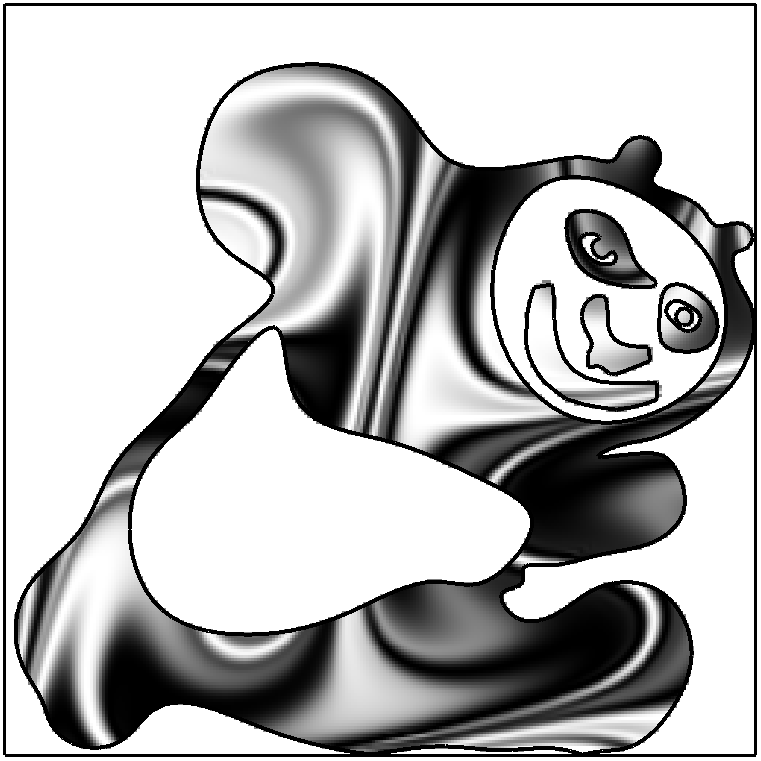}
  }
  \subfigure[$t=0.8$]{
    \includegraphics[width=0.3\textwidth, height=35mm]
    {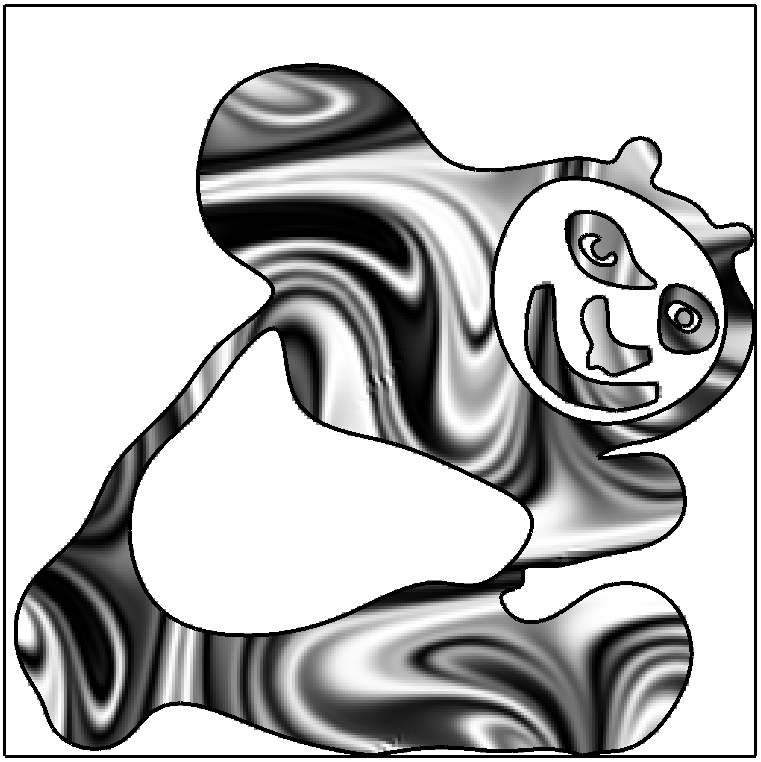}
  }
  \subfigure[$t=1$]{
    \includegraphics[width=0.3\textwidth, height=35mm]
    {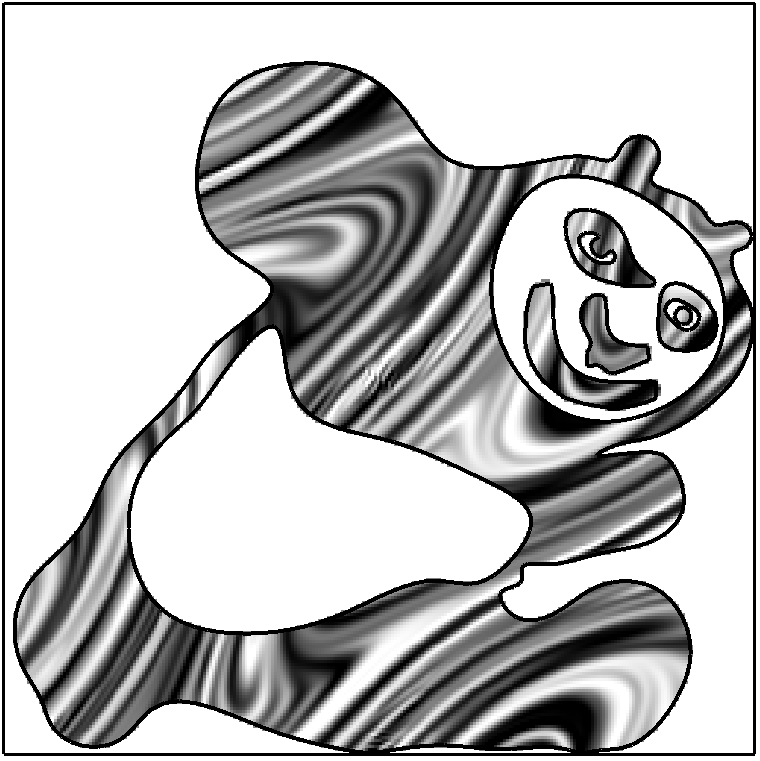}
  }
  \caption{Results of the SLFV method for the chaotic mixing test
    in \Cref{sec:penetr-veloc-bound-5}
    on a complex irregular domain with $h = \frac{1}{256}$. 
    The lightest color corresponds to the value of $1$ 
    and the darkest color corresponds to the value of $-1$.
    In each snapshot, the solution over the panda is supposed to be
    the same as that in the corresponding subplot
    in \Cref{fig:densityEvolutionoptimalControlOfMixing}. 
  } 
\label{fig:densityEvolutionoptimalControlOfMixingOnPanda}
\end{figure}

\begin{table}
  \centering
  \caption{Errors and convergence rates of 
    the $\kappa$-th order SLFV method 
    for the chaotic mixing test in \Cref{sec:penetr-veloc-bound-5}
    with $T=0.1$ and $k=0.4h$.
    The errors are calculated by Richardson extrapolation.
  } 
  \label{table:pandaPenetrationIncompressibleZeroSourceTermLargeVortex}
  \begin{tabular}{c|ccccccccc}
    \hline    
    $\kappa$ && $\frac{1}{256}$--$\frac{1}{512}$ & rate & 
                $\frac{1}{512}$--$\frac{1}{1024}$ & rate & 
                $\frac{1}{1024}$--$\frac{1}{2048}$\\
    \hline\hline
    \multirow{3}{*}{4} & $L^\infty$ & 
    6.14e-07 & 3.70 & 4.73e-08 & 3.98 & 3.00e-09 \\
    ~ & $L^1$ & 
    8.82e-08 & 3.90 & 5.92e-09 & 4.05 & 3.58e-10 \\
    ~ & $L^2$ & 
    1.36e-07 & 3.88 & 9.26e-09 & 4.05 & 5.60e-10 \\
    \hline\hline
    \multirow{3}{*}{6} & $L^\infty$ & 
    9.96e-09 & 5.80 & 1.79e-10 & 6.09 & 2.62e-12 \\
    ~ & $L^1$ & 
    1.26e-09 & 5.75 & 2.34e-11 & 5.92 & 3.88e-13 \\
    ~ & $L^2$ & 
    1.84e-09 & 5.68 & 3.58e-11 & 5.94 & 5.84e-13 \\
    \hline
  \end{tabular}
\end{table}

In \cref{fig:densityEvolutionoptimalControlOfMixingOnPanda}, 
 we present snapshots of the solution
 in the same manner and at the same instants
 with those in \cref{fig:densityEvolutionoptimalControlOfMixing}; 
 any two corresponding subplots
 are visually indistinguishable over the panda region, 
 which indicates an accurate handling
 of the complex domain.
As more quantitative evidences, 
 errors and convergence rates of the solution $\rho$ at $T=0.1$ 
 are presented in 
 \cref{table:pandaPenetrationIncompressibleZeroSourceTermLargeVortex},
 clearly demonstrating
 the fourth- and sixth-order convergence rates of our SLFV method
 over this highly irregular domain.

\section{Conclusion}
\label{sec:conclusion-1}
We have proposed fourth- and higher-order SLFV methods
 for solving the two-dimensional advection equation 
 on both regular and irregular domains
 with periodic and incoming penetration conditions.
The main components of SLFV
 are the Yin sets for representing domains
 with arbitrarily complex topology and geometry,
 Gauss quadrature formulas on quadrilaterals/triangles,
 and a Newton iteration algorithm
 for intersecting a pathline with the domain boundary.
The orthogonality of these components
 furnishes flexibility, user friendliness,
 and the ease of implementation.
Results of various numerical tests demonstrate
 the effectiveness of SLFV,
 its excellent conditioning, 
 and its high-order convergence rates both in time and in space.

To prove the high-order convergence rates of SLFV,
 we have assumed the smoothness of the solution
 of the advection equation.
If the source term in \cref{eq:advection_control_equation_a}
 vanishes and the initial condition
 in \cref{eq:advection_control_equation_init}
 contains discontinuities,
 the high-order convergence rates can be maintained
 by coupling this work to a high-order interface tracking method, 
 such as the MARS method in
 \cite{zhang18:_cubic_mars_method_fourt_sixth}, 
 which tracks the loci of discontinuity
 on a one-dimensional CW-complex. 
We will report this extension in a future paper.

\revise{
  It should be straightforward
  to extend the SLFV method
  to solving the advection equation \cref{eq:advection_control_equation}
  on moving domains.
The key distinction 
lies in the treatment of incoming penetration conditions: 
 instead of computing the intersection point $\mathbf{x}^*$ of 
 the particle pathline with the fixed domain boundary 
 in $\mathbb{R}^2$,
 for moving boundaries
 we determine the space-time intersection point $(\mathbf{x}^*,t^*)$ 
 in the extended phase space $\mathbb{R}^2\times\mathbb{R}$
 and set the value of the corresponding quadrature point
 to be the solution of \cref{eq:advection_control_equation_a}
 whose initial condition $\rho(\mathbf{x}^*,0)$
 is the boundary condition at $(\mathbf{x}^*,t^*)$. 
}
 
We also plan to 
 \revise{extend the SLFV method to the nonlinear advection equation 
 and validate its efficacy through the guiding-center model 
 \cite{crouseilles2010conservative, crouseilles2012discontinuous}. 
Subsequently, we will}
 couple the SLFV method with high-order methods 
 for incompressible Navier-Stokes equations
 \cite{Zhang2016GePUP,Li2024GePUP-E}
 to form a more sophisticated solver, 
 which could be useful in studying physical processes 
 governed by the Boussinesq equations.

\section*{Acknowledgments}
\revise{We are grateful to two anonymous referees 
 for their insightful comments and suggestions.}
We also thank helpful comments from Shuang Hu and Chenhao Ye, 
 graduate students at the school of mathematical sciences
 in Zhejiang University.
 
\bibliographystyle{siamplain}
\bibliography{advectionSLFV}

\end{document}